\documentclass[uplatex,12pt]{article}
\usepackage{amsmath}
\usepackage{amssymb}
\usepackage{latexsym}
\usepackage[dvips]{color}
\usepackage{fancybox}
\usepackage{graphics}
\usepackage{amsthm}
\usepackage{tabularx}
\usepackage{amsthm}
\usepackage{subcaption}
\usepackage{mathrsfs} 
\usepackage{colortbl}
\usepackage{tikz-cd}
\usepackage{bm}
\usepackage{hyperref}
\usepackage{textcomp}%
\usepackage{manyfoot}%
\usepackage{booktabs}%
\usepackage{algorithm}%
\usepackage{algorithmicx}%
\usepackage{algpseudocode}%
\usepackage{listings}%

\usepackage[a4paper]{geometry}
\usepackage{mathrsfs}
\usepackage{comment}
\usepackage{accsupp}

\def\supp{\mathop{\mathrm{supp}}}
\def\diam{\mathop{\mathrm{diam}}}
\def\diag{\mathop{\mathrm{diag}}}

\def\div{\mathop{\mathrm{div}}}

\def\<{\mathop{\textless}}
\def\>{\mathop{\textgreater}}

\def\argmin{\mathop{\mathrm{arg \ min}}}

\theoremstyle{definition}
\newtheorem{thr}{Theorem}[section]

\newtheorem*{ex*}{Example}
\newtheorem{rem}[thr]{Remark}
\newtheorem{defi}[thr]{Definition}
\newtheorem{lem}[thr]{Lemma}

\newtheorem*{pf*}{Proof}

\newtheorem{cor}[thr]{Corollary}
\newtheorem{assume}{Assumption}{\bf}{\it}
\newtheorem*{ass*}{Assumption}
\newtheorem{Cond}[thr]{Condition}

\newtheorem*{exer*}{Exercise}

\newtheorem*{memo*}{Memo}

\newtheorem*{pfth*}{Proof of Theorem 2.6}
\newtheorem*{pfth1*}{Proof of Theorem 2.7}

\newtheorem*{lem21*}{Lemma 2.1}
\newtheorem*{th15*}{Theorem 1.5}
\newtheorem*{th21*}{Theorem 2.1}
\newtheorem*{th25*}{Theorem 2.5}
\newtheorem*{th26*}{Theorem 2.6}

\usepackage{makeidx}
\makeindex

\allowdisplaybreaks[1]

\newcounter{sone}
\newcounter{stwo}
\newcounter{sthree}
\newcounter{sfour}
\newcounter{sfive}
\newcounter{ssix}
\newcounter{lone}
\newcounter{ltwo}
\newcounter{lthree}
\newcounter{lfour}
\newcounter{lfive}
\newcounter{lsix}
\newcounter{lseven}
\newcounter{leight}
\makeatletter
    
    \@addtoreset{equation}{section}
  \makeatother

\begin{document}

\setcounter{section}{0}
\title{On discrete Sobolev inequalities for nonconforming finite elements under a semi-regular mesh condition}
\author{Hiroki Ishizaka \thanks{Team FEM, Matsuyama, Japan,  h.ishizaka005@gmail.com}}

\date{\today}

\maketitle
\pagestyle{plain}

\begin{abstract}
We derive a discrete $ L^q-L^p$ Sobolev inequality tailored for the Crouzeix--Raviart and discontinuous Crouzeix--Raviart finite element spaces on anisotropic meshes in both two and three dimensions. Subject to a semi-regular mesh condition, this discrete Sobolev inequality is applicable to all pairs $(q,p)$ that align with the local Sobolev embedding, including scenarios where $q \leq p$. Importantly, the constant is influenced solely by the domain and the semi-regular parameter, ensuring robustness against variations in aspect ratios and interior angles of the mesh. The proof employs an anisotropy-sensitive trace inequality that leverages the element height, a two-step affine/Piola mapping approach, the stability of the Raviart--Thomas interpolation, and a discrete integration-by-parts identity augmented with weighted jump/trace terms on faces. This Sobolev inequality serves as a mesh-robust foundation for the stability and error analysis of nonconforming and discontinuous Galerkin methods on highly anisotropic meshes.


\end{abstract}

\section{Introduction}
\label{sec;introduction}
This paper introduces discrete Sobolev inequalities applicable to nonconforming finite elements on semi-regular (anisotropic) meshes. Specifically, for the Crouzeix--Raviart (CR) finite element space and its discontinuous counterpart, we establish an $L^q - L^p$ ($q \leq p$) bound. {Significantly, the constants associated with this bound are exclusively determined by the dimension, the domain, the exponents $p$ and $q$, and the semi-regular mesh parameter. These constants remain unaffected by the mesh size, as well as the aspect ratios and interior angles of the simplices.} {This extends the classical shape-regular theory (e.g. \cite{Tem74,Tem77,Tho77,Bre03,PieErn12})} to anisotropic partitions and provides a robust tool for stability and error analysis.

Let $\Omega \subset \mathbb{R}^d$, $d \in \{2,3\}$, be a bounded Lipschitz polyhedral domain. {The limitation to $d \in \{2,3\}$ arises from the application of anisotropic Raviart--Thomas (RT) interpolation error estimates on semi-regular meshes as detailed in \cite{Ish22}. These estimates are derived through a two-step affine-transformation technique (combined with the associated Piola transform for RT fields) that employs an explicit curvilinear parametrisation of the geometry of triangles and tetrahedra. Extending this framework to $d \geq 4$ would necessitate a corresponding higher-dimensional geometric parametrisation (or an alternative coordinate-free reformulation) along with new anisotropic bounds for the RT facet moments; such results are currently not available. We leave this extension as a challenging avenue for future research. In contrast, the Bogovski\u{\i}-type right inverse of the divergence, as well as the Poincar\'e and trace inequalities used in this paper, are applicable in any spatial dimension.}
As a representative model problem that motivates the discrete Sobolev inequalities proved below, we recall the Poisson problem: find $u: \Omega \to \mathbb{R}$ such that
\begin{equation}\label{eq:Poisson}
  -\Delta u = f \quad\text{in }\Omega, 
  \qquad u = 0 \quad\text{on }\partial\Omega,
\end{equation}
where $f \in L^2(\Omega)$ is a given function.

For nonconforming and discontinuous Galerkin discretisations of \eqref{eq:Poisson}, discrete Poincar\'e--Sobolev inequalities of the type studied in this paper lead to mesh-robust stability bounds of the form
\begin{align*}
\displaystyle
|u_h|_E \leq C^{\mathrm{stab}} \| f \|_{L^2(\Omega)}.
\end{align*}
where $|\cdot|_E$ denotes the natural energy norm on the discrete space and 
$C^{\mathrm{stab}}$ is independent of the mesh size $h$ and of the aspect ratios and interior angles of the simplices; see, for example, the CR-type schemes analysed in \cite{Ish24a,Ish25b}. This provides one of the main motivations for deriving discrete $L^q - L^p$ Sobolev inequalities under a semi-regular mesh condition in the present work. {Discrete Sobolev embeddings of the $L^q - L^p$ type are a fundamental tool in the analysis of nonlinear problems. For example, they are integral to the stability analysis of nonconforming discretisations of $p$-Laplacian-type models, and more broadly, problems involving nonlinear lower-order terms. In such contexts, it is essential to have a mesh-robust embedding estimate $\| u_h \|_{L^q(\Omega)} \leq C^{\mathrm{pLap}} |u_h|_{W^{1,p}(\mathbb{T}_h)}$, where $C^{\mathrm{pLap}}$ remains independent of the mesh size $h$, as well as the aspect ratios and interior angles of the simplices, to control the nonlinear terms. Similarly, in the case of the stationary incompressible Navier--Stokes equations, these embeddings are employed to estimate the convection term, specifically the associated trilinear form, in terms of broken $W^{1,p}$ seminorms.}


\color{black}

For piecewise $H^1$ functions on general partitions, the Poincar\'e--Friedrichs inequalities are valid with constants that depend solely on the shape-regularity of the partition, without necessitating quasi-uniformity. In two dimensions ($2D$), shape-regularity can be assessed through triangulations that adhere to certain criteria, such as a minimum-angle parameter. In three dimensions ($3D$), it is determined by geometric parameters in conjunction with a uniform face-angle bound, as detailed in \cite{Bre03}.

Discrete Poincar\'e- and Sobolev-type inequalities for nonconforming
CR-type finite element spaces have, in fact, a longer history.
Early instances of such estimates can already be found in the works of Temam
and Thomas in the context of stationary Navier--Stokes equations and mixed or
nonconforming finite element approximations on shape-regular meshes; see,
for example, \cite{Tem74,Tem77,Tho77}. \cite{Bre03} later provides
a systematic treatment of Poincar\'e--Friedrichs inequalities for piecewise
$H^1$ functions on shape-regular partitions {including nonconforming partitions with hanging nodes and edges}.

Anisotropic finite elements, under mesh assumptions related to the classical maximum-angle condition, have been extensively investigated. In \cite{Ape99}, Apel established a comprehensive framework for anisotropic finite elements, offering precise local interpolation estimates and geometric tools for addressing anisotropic simplicial meshes. Building upon this framework, \cite{ApeNicSch01} conducted an in-depth analysis of CR-type finite elements on anisotropic meshes in three dimensions. They derived optimal-order finite element error estimates for a nonconforming approximation of the Poisson problem by integrating anisotropic interpolation bounds with consistency error estimates. To the best of our knowledge, neither \cite{Ape99} nor \cite{ApeNicSch01} explicitly formulates global discrete $L^q - L^p$ Sobolev or Poincar\'e inequalities for CR-type spaces as a primary focus.

Demonstrating the discrete Sobolev (Poincar\'e) inequality on anisotropic meshes presents significant challenges. Recent research has made notable progress in this area. For the discontinuous Galerkin-type finite element space underlying an anisotropic weakly over-penalised symmetric interior penalty (WOPSIP) formulation of the Stokes equations, \cite{Ish24a} established a discrete Poincar\'e inequality on anisotropic semi-regular meshes. For the CR finite element space, \cite{Ish25a} derived discrete $L^q - L^p$ Sobolev inequalities on anisotropic meshes under a semi-regular mesh condition and the weak elliptic regularity assumption. In particular, when $p = q = 2$, the proof in \cite{Ish25a} uses a duality argument posed on a convex domain; see also the references therein for related variants such as hybrid interior penalty and Nitsche-type formulations. These investigations rely on duality-based proofs, for which the convexity of $\Omega$ is crucial in order to validate the inequalities. The present paper develops an alternative approach based on the Bogovski\v{i} operator and RT interpolation, which applies to general bounded Lipschitz polyhedral domains and therefore removes the convexity restriction. {In the context of the Hilbert space where $p=q=2$, related Poincar\'e--Friedrichs-type estimates concerning general decompositions under appropriate trace inequalities can be traced back at least to \cite[Lemma 2.1]{Arn82}. Our findings also address the scenario where $p=q=2$, and we remove the convexity restrictions typically associated with duality and elliptic regularity arguments by utilising the Bogovski\v{i}/RT-based approach on bounded Lipschitz polyhedral domains.}

{Our findings expand the application of discrete Poincar\'e--Friedrichs and Sobolev inequalities to piecewise $H^1$ functions on shape-regular partitions. This expansion includes nonconforming partitions with hanging nodes and edges, as described in \cite{Bre03}, and further extends these inequalities to anisotropic semi-regular simplicial meshes.} The semi-regular mesh condition employed in this paper is expressed in terms of a geometric flatness parameter; for tetrahedral meshes, it was shown in \cite{IshKobSuzTsu21} that this semi-regular condition is in fact equivalent to the classical maximum-angle condition. 
\color{black} 
For $q \> p$, an $L^p-L^2$ estimate is available under a weak elliptic regularity (see Section \ref{Con=rem} and \cite[Lemma 4]{Ish25a}).

Section 2 establishes the notation and anisotropic preliminaries. Section 3 introduces semi-regular meshes and the RT interpolation. Section 4 compiles the Bogovski\u{\i} operator, discrete integration-by-parts, and face estimates. Section 5 demonstrates the primary discrete $L^q-L^p$ inequality, while Section 6 provides concluding remarks.

{Throughout this paper, we denote by $c$ generic constants, independent of $h$ (defined later) and the angles and aspect ratios of simplices; unless specified otherwise, all constants $c$ are bounded if the maximum angle is bounded.} These values vary across different contexts. The notation $\mathbb{R}_+$ denotes a set of positive real numbers.

\section{Preliminaries}
{In this section, we collect the notation, mesh assumptions and anisotropic preliminaries used in the subsequent analysis.}

\label{Preliminaries=sec}

\subsection{Mesh, mesh faces, jumps and averages} \label{sec=mesh}

Let $\mathbb{T}_h$ be a simplicial mesh of $\overline{\Omega}$ composed of closed $d$-simplices. Let $N_e := \#\mathbb{T}_h$ denote the number of elements of the mesh and write
\begin{align*}
\displaystyle
\mathbb{T}_h = \{ T_j \}_{j=1}^{N_e}, \quad \overline{\Omega} = \bigcup_{j=1}^{N_e} T_j.
\end{align*}
For each $j \in \{1,\ldots,N_e\}$, we set
\begin{align*}
\displaystyle
h_{T_j} := \diam(T_j), \quad h := \max_{1 \le j \le N_e} h_{T_j}.
\end{align*}
{For simplicity, we assume that $\mathbb{T}_h$ is conforming; that is, $\mathbb{T}_h$ is a simplicial mesh of $\overline{\Omega}$ with no hanging nodes. Throughout this paper, we assume $0 < h \leq 1$.}

\begin{rem}[Conforming meshes]
The assumption of a conforming mesh (i.e., the absence of hanging nodes) is made solely to simplify the face-based notation and constructions. In particular, it ensures that each interior face $F$ is a single $(d-1)$-simplex shared by exactly two elements, so that the unit normal and the jump and average operators on $F$ are unambiguous. This assumption also streamlines the definition of global face-based finite element spaces, such as the CR-type nonconforming spaces and the $H(\div)$-conforming RT space, where face degrees of freedom (including normal-flux moments for RT) are matched consistently across each interior face.

Meshes with hanging nodes can, in principle, be treated by working on the induced subdivision of the mesh skeleton into subfaces and by redefining the face-based operators and matching conditions accordingly; for CR-type settings, such constructions are classical, see e.g., \cite[Section 2.7]{Bre15}. For the $H(\div)$-conforming RT space, however, hanging faces typically require additional design choices and constraints to enforce consistent matching of normal-flux moments across a coarse face and its subfaces (while preserving the standard commuting and approximation properties). To avoid this additional technical overhead, we restrict ourselves to conforming meshes in the present paper.

\end{rem}

\color{black}
Let $p \in [1,\infty)$ and $s \> 0$ be a positive real number. We define a broken (piecewise) Sobolev space as
\begin{align*}
\displaystyle
W^{s,p}(\mathbb{T}_h) := \{ v \in L^p(\Omega): \ v|_{T} \in W^{s,p}(T) \ \forall T \in \mathbb{T}_h  \}
\end{align*}
{equipped with the broken Sobolev seminorm}
\begin{align*}
\displaystyle
| v |_{W^{s,p}(\mathbb{T}_h)} &:= \left( \sum_{T \in \mathbb{T}_h} | v |^p_{W^{s,p}(T) } \right)^{\frac{1}{p}} \quad \text{if $p \in [1,\infty)$}.
\end{align*}
{In particular}, we write $H^s(\mathbb{T}_h) := W^{s,2}(\mathbb{T}_h)$ {and}
\begin{align*}
\displaystyle
| v |_{H^s(\mathbb{T}_h)} &:= \left( \sum_{T \in \mathbb{T}_h}| v |^2_{H^s(T)} \right)^{\frac{1}{2}} \quad v \in H^s(\mathbb{T}_h).
\end{align*}

We remark that in the present paper, we only use the broken Sobolev space $W^{s,p}(\mathbb{T}_h)$ with the integer order $s=1$, i.e. $W^{1,p}(\mathbb{T}_h)$ and $H^1(\mathbb{T}_h) = W^{1,2}(\mathbb{T}_h)$. In particular, the Bogovski\u{\i}-type right-inverse of the divergence in Lemma~\ref{lem:Bogovskii} is applied only in the classical setting $L^r_0(\Omega) \to  W^{1,r}_0(\Omega)^d$ for $r \in (1,\infty)$. Thus, our use of the Bogovski\u{\i} operator is restricted to the standard integer-order setting, and no results for fractional-order Sobolev spaces are needed.

Furthermore, throughout the present paper, we restrict the exponent to $p \in [1,\infty)$. This is the natural range in which we use the trace theorem for $W^{1,p}$ functions on Lipschitz domains (see, e.g., \cite[Theorem 3.10]{ErnGue21a}), in order to give a precise meaning to face traces and thus to jumps and averages. Since our discrete $L^q - L^p$ Sobolev inequalities in Section \ref{sec;proof=main} are also formulated only for finite exponents $p$ and $q$, the endpoint case  $p = \infty$ is not needed in the sequel.

Let $\mathcal{F}_h^i$ be the set of interior faces, and $\mathcal{F}_h^{\partial}$ be the set of faces on the boundary $\partial \Omega$. We set $\mathcal{F}_h := \mathcal{F}_h^i \cup \mathcal{F}_h^{\partial}$. For any $F \in \mathcal{F}_h$, we define the unit normal $\bm n_F$ to $F$ as follows: (\roman{sone}) If  $F \in \mathcal{F}_h^i$ with $F = \partial T_{+} \cap \partial T_{-}$, $T_{+},T_{-} \in \mathbb{T}_h$, $+ > -$, let $\bm n_F$ be the unit normal vector from $T_{+}$ to  $ T_{-}$.  (\roman{stwo}) If $F \in \mathcal{F}_h^{\partial}$, $\bm n_F$ is {the unit outward normal} $\bm n$ to $\partial \Omega$. We fix once and for all a local ordering of the neighbouring elements and use the symbol "$+ > -$" to indicate that $T_+$ is the first element and $T_-$ the second one in this ordering.

\color{black}
Let $p \in [1,\infty)$. Let $\varphi \in W^{1,p}(\mathbb{T}_h)$. Suppose that $F \in \mathcal{F}_h^i$ with {$F = \partial T_{+} \cap \partial T_{-}$}, $T_{+},T_{-} \in \mathbb{T}_h$. For such a face, we set  $\varphi_{+} := \varphi{|_{T_{+}}}$ and $\varphi_{-} := \varphi{|_{T_{-}}}$.  We choose two nonnegative real numbers $\omega_{T_{+},F}$ and $\omega_{T_{-},F}$ with
\begin{align*}
\displaystyle
\omega_{T_{+},F} + \omega_{T_{-},F} = 1.
\end{align*}
The jump and skew-weighted averages of $\varphi$ across $F$ are then defined as
\begin{align*}
\displaystyle
[\![\varphi]\!] := [\! [ \varphi ]\!]_F := \varphi_{+} - \varphi_{-}, \quad  \{\! \{ \varphi\} \! \}_{\overline{\omega}} :=  \{\! \{ \varphi\} \! \}_{\overline{\omega},F} := \omega_{T_{-},F} \varphi_{+} + \omega_{T_{+},F} \varphi_{-}.
\end{align*}
For a boundary face $F \in \mathcal{F}_h^{\partial}$ with $F = \partial T \cap \partial \Omega$, $[\![\varphi ]\!]_F := \varphi|_{T}$ and $\{\! \{ \varphi \} \!\}_{\overline{\omega}} := \varphi |_{T}$. For any ${\bm {v}} \in W^{1,p}(\mathbb{T}_h)^d$, the notation
\begin{align*}
\displaystyle
&[\![{\bm {v}} \cdot {\bm {n}}]\!] := [\![ {\bm {v}} \cdot {\bm {n}} ]\!]_F := {\bm {v}_{+}} \cdot {\bm {n}_F} - {\bm {v}_{-}} \cdot {\bm {n}_F},\\
&  \{\! \{ {\bm {v}}\} \! \}_{\omega} :=  \{\! \{ {\bm {v}} \} \! \}_{\omega,F} := \omega_{T_{+},F} {\bm {v}_{+}} + \omega_{T_{-},F} {\bm {v}_{-}}
\end{align*}
for the jump in the normal component and the weighted average of ${\bm {v}}$. For any ${\bm {v}} \in W^{1,p}(\mathbb{T}_h)^d$ and $\varphi \in W^{1,p}(\mathbb{T}_h)$, we have that
\begin{align*}
\displaystyle
[\![ ({\bm {v}} \varphi) \cdot {\bm {n}} ]\!]_F
&=  \{\! \{ {\bm {v}} \} \! \}_{\omega,F} \cdot {\bm {n}_F} [\! [ \varphi ]\!]_F + [\![ {\bm {v}} \cdot {\bm {n}} ]\!]_F \{\! \{ \varphi\} \! \}_{\overline{\omega},F}.
\end{align*}
This clarifies the meaning of  "$+ > -$" wherever it is used in the manuscript.

\color{black}
\subsection{Trace inequality} \label{trace=sec}
The trace inequality on anisotropic meshes discussed herein is of considerable importance to this study. The proof of this inequality is documented in several references. In this context, we adhere to the approach outlined by Ern and Guermond \cite[Lemma 12.15]{ErnGue21a}. It is noteworthy that, although Lemma 12.15 in \cite{ErnGue21a} stipulates a shape-regular mesh condition, the condition is easily violated. For a simplex $T \subset \mathbb{R}^d$, let $\mathcal{F}_{T}$ be the collection of the faces of $T$. Let $|\cdot|_d$ denote the $d$-dimensional Hausdorff measure.

\begin{lem}[Trace inequality] \label{lem=trace}
Let $p \in [1,\infty]$. Let  $T \subset \mathbb{R}^d$ be a simplex. There exists a positive constant ${C^{Tr}(d,p)}$ such that for any $\bm v = (v^{(1)}, \ldots,v^{(d)})^{\top}  \in W^{1,p}(T)$, $F \in \mathcal{F}_{T}$, and $h$,
\begin{align}
\displaystyle
\| \bm v \|_{L^p(F)^d}
\leq {C^{Tr}(d,p)} \ell_{T,F}^{- \frac{1}{p}} \left( \| \bm v \|_{L^p(T)^d} + h_{T}^{\frac{1}{p}}  \| \bm v \|_{L^p(T)^d}^{1- \frac{1}{p}} | \bm v |_{W^{1,p}(T)^d}^{\frac{1}{p}} \right),\label{trace}
\end{align}
where {$\ell_{T,F} := \frac{d |T|_d}{|F|_{d-1}}$} denotes the distance of the vertex of $T$ opposite to $F$ to the face. {Here, the constant $C^{Tr}(d,p)$ depends only on the dimension $d$ and the exponent $p$ and is independent of $T$, $F$, the mesh size $h_T$ and the semi-regular parameter $\gamma_0$.}
\end{lem}

\begin{pf*}
The estimate is a direct consequence of the anisotropic trace inequality on simplices proved in \cite[Lemma~12.15]{ErnGue21a}. More precisely, applying \cite[Lemma~12.15]{ErnGue21a} to each component $v^{(i)}$ of $\bm v = (v^{(1)}, \ldots,v^{(d)})^{\top}$ and expressing the resulting bound in terms of the height {$\ell_{T,F} =\frac{d |T|_d}{|F|_{d-1}}$} yields the scalar version of \eqref{trace} for $v^{(i)}$. Summing over $i=1,\ldots,d$ and using the Cauchy--Schwarz inequality then gives \eqref{trace} with a constant depending only on $d$ and $p$. We refer to \cite[Lemma~12.15]{ErnGue21a} for the detailed scalar proof.
\end{pf*}

\color{black}
\begin{rem}
Because $|T|_d \approx h_T^d$ and $|F|_{d-1} \approx h_T^{d-1}$ on the shape-regular mesh, it holds that $\ell_{T,F} \approx h_T$. Then, the trace inequality \eqref{trace} is given as
\begin{align*}
\displaystyle
\| v \|_{L^p(F)}
\leq c h_T^{- \frac{1}{p}} \left( \| v \|_{L^p(T)} + h_{T}^{\frac{1}{p}}  \| v \|_{L^p(T)}^{1 - \frac{1}{p}} | v |_{W^{1,p}(T)}^{\frac{1}{p}} \right).
\end{align*}
\end{rem}

\subsection{Finite element spaces} \label{CRspace}
For any $T \in \mathbb{T}_h$ and $F \in \mathcal{F}_h$, let $D \in \{ T , F \}$. For $k \in \mathbb{N}_0 := \mathbb{N} \cup \{ 0 \}$, $\mathbb{P}^k(D)$ is spanned by the restriction to $D$ of polynomials in $\mathbb{P}^k$, where  $\mathbb{P}^k$ denotes the space of polynomials with a maximum of $k$ degrees.

For $k \in \mathbb{N}_0$, we define the standard discontinuous finite-element space as
\begin{align}
\displaystyle
V_{h}^{DC(k)} &:= \left\{ v_h \in {L^{\infty}(\Omega)}; \ v_h|_{T} \in \mathbb{P}^{k}({T}) \quad \forall T \in \mathbb{T}_h  \right\}. \label{dis=sp}
\end{align}

In this paper, we focus exclusively on the lowest-order cases, specifically $k=0$ and $k=1$. The space $V_h^{\mathrm{DC}(0)}$, consisting of piecewise constant functions, is utilised through the elementwise $L^2$-projection $\Pi_{h}^0:L^1(\Omega)\to V_h^{\mathrm{DC}(0)}$. Meanwhile, the CR spaces $V_h^{\mathrm{CR}}$, $V_{h0}^{\mathrm{CR}}$, and the discontinuous CR space $V_h^{\mathrm{DCCR}}$, which are introduced subsequently, are subspaces of $V_h^{\mathrm{DC}(1)}$. The discrete $L^q - L^p$ Sobolev inequalities demonstrated in Section \ref{sec;proof=main} are applicable solely to these lowest-order spaces. {The limitation to $k=0,1$ is not solely a matter of notation. Extending the current framework to encompass higher-order CR-type elements ($k \geq 2$) on semi-regular anisotropic meshes necessitates the development of a corresponding higher-order CR-type theory. This would include anisotropic approximation and interpolation error estimates based on the parameters of the semi-regular mesh. As such results are presently unavailable within the semi-regular anisotropic context under consideration, we do not pursue this extension.}


\color{black}

{Recalling that $\mathbb{T}_h = \{ T_j\}_{j=1}^{Ne}$}, we introduce a discontinuous CR finite element space.

Let the points $\{ \bm P_{T_j,1}, \ldots, \bm P_{T_j,d+1} \}$ be the vertices of the simplex $T_j \in \mathbb{T}_h$ for $j \in \{1, \ldots , Ne \}$. Let $F_{T_j,i}$ be the face of $T_j$ opposite $\bm P_{T_j,i}$ for $i \in \{ 1, \ldots , d+1\}$. We take a set ${\Sigma}_{T_j} := \{ {\chi}^{CR}_{T_j,i} \}_{1 \leq i \leq d+1}$ of linear forms with its components such that for any $q \in \mathbb{P}^1$.
\begin{align}
\displaystyle
{\chi}^{CR}_{T_j,i}({q}) := \frac{1}{| {F}_{T_j,i} |_{d-1}} \int_{{F}_{T_j,i}} {q} d{s} \quad \forall i \in \{ 1, \ldots,d+1 \}. \label{CR1}
\end{align}
For each $j \in \{1, \ldots ,Ne \}$, the triple $\{ T_j ,  \mathbb{P}^1 , \Sigma_{T_j} \}$ is a finite element. Using the barycentric coordinates $ \{{\lambda}_{T_j,i} \}_{i=1}^{d+1}: \mathbb{R}^d \to \mathbb{R}$ on the reference element, the nodal basis functions associated with the degrees of freedom by \eqref{CR1} are defined as
\begin{align}
\displaystyle
{\theta}^{CR}_{T_j,i}(\bm {x}) := d \left( \frac{1}{d} - {\lambda}_{T_j,i} (\bm {x}) \right) \quad \forall i \in \{ 1, \ldots ,d+1 \}. \label{CR2}
\end{align}
For $j \in \{1, \ldots ,Ne \}$ and $i \in \{ 1, \ldots , d+1\}$, we define the function $\phi_{j(i)}$ as
\begin{align}
\displaystyle
\phi_{j(i)}(\bm x) :=
\begin{cases}
\theta_{T_j,i}^{CR}(\bm x), \quad \text{$\bm x \in T_j$}, \\
0, \quad \text{$x \notin T_j$}.
\end{cases} \label{CR5}
\end{align}
We define a discontinuous CR finite element space as
\begin{align}
\displaystyle
V_{h}^{DCCR} &:= \left\{ \sum_{j=1}^{Ne} \sum_{i=1}^{d+1} c_{j(i)} \phi_{j(i)}; \  c_{j(i)} \in \mathbb{R}, \ \forall i,j \right\} {=} V_{h}^{DC(1)}. \label{CR6}
\end{align}
Furthermore, we define standard CR finite element spaces as
\begin{align}
\displaystyle
V_{h}^{CR} &:= \biggl \{ \varphi_h \in V_{h}^{DCCR}: \  \int_F [\![ \varphi_h ]\!] ds = 0 \ \forall F \in \mathcal{F}_h^i \biggr \}, \label{CRdef} \\
V_{h0}^{CR} &:= \biggl \{ \varphi_h \in V_{h}^{CR}: \  \int_F \varphi_h  ds = 0 \ \forall F \in \mathcal{F}_h^{\partial} \biggr \}. \label{CR0def}
\end{align}
We define $V_h \in \{V_{h}^{DCCR}, V_h^{CR},V_{h0}^{CR} \}$.

\begin{rem}
On each element $T_j$, the local CR basis functions $\{ \theta_{T_j,i}^{CR} \}_{i=1}^{d+1}$ form a basis of $\mathbb{P}^{1}({T}_j)$, and $\phi_{j(i)}$ are exactly these functions extended by zero outside $T_j$. Therefore, $V_{h}^{DCCR} = V_{h}^{DC(1)}$. We keep the notation $V_{h}^{DCCR}$ to emphasise that we use CR-type degrees of freedom and basis functions on each element.
\end{rem}

\color{black}
\subsection{Norms} \label{CRnorms}
Let $F \in \mathcal{F}_h^i$ with {$F = \partial T_{+} \cap \partial T_{-}$}, $T_{+},T_{-} \in \mathbb{T}_h$, $+ > -$ be an interior face and $F \in \mathcal{F}_h^{\partial}$ with $F = \partial T_{\partial} \cap \partial \Omega$, $T_{\partial} \in \mathbb{T}_h$ a boundary face. Let $p \in [1,\infty)$. A choice for the weighted parameters is such that
\begin{align}
\displaystyle
 \omega_{T_i,F} :=  \frac{{\ell_{T_i,F}^{\frac{p-1}{p}}}}{{\ell_{T_{+},F}^{\frac{p-1}{p}}} + {\ell_{T_{-},F}^{\frac{p-1}{p}}}}, \quad i \in \{+,- \}. \label{weight}
\end{align}

The weights $ \omega_{T_i,F}$ satisfy $ \omega_{T_i,F} \in [0,1]$ and $ \omega_{T_+,F} +  \omega_{T_-,F} = 1$  and they depend only on the face heights $\ell_{T_i,F}$. They will be used in the definition of weighted averages on interior faces so that the contribution of each neighbouring element is tuned according to its height in the normal direction to $F$.

\color{black}
For the proof of the discrete Sobolev inequality, we will use the following parameter.
\begin{align}
\displaystyle
\kappa_{p,F*} :=
\begin{cases}
\displaystyle
\left( { \ell_{T_{+},F}^{\frac{p-1}{p}}} + {\ell_{T_{-},F}^{\frac{p-1}{p}}} \right)^{-p} \quad \text{if $F \in \mathcal{F}_h^i$},\\
\displaystyle
\ell_{T_{\partial},F}^{1-p} \quad \text{if $F \in \mathcal{F}_h^{\partial}$}.
\end{cases} \label{penalty1}
\end{align}

where $\ell_{T,F}$ is the height of $T$ with respect to $F$ in Section \ref{trace=sec}. This parameter is constructed solely from the local face heights \(\ell_{T,F}\) and plays the role of an inverse normal mesh size on each face. On a regular (shape-regular and quasi-uniform) mesh, one has \(\kappa_{p,F}^* \approx h_F^{-1} \approx h^{-1}\), so that \(\kappa_{p,F}^*\) grows as the mesh is refined. For an explicit formula and several examples illustrating its behaviour on anisotropic meshes, we refer to \cite[Section 2.4]{Ish24a}.

\color{black}
Let $p \in [1,\infty)$ and $V_h \in \{V_{h}^{DCCR}, V_h^{CR},V_{h0}^{CR} \}$. {We define the following mesh-dependent functionals for any $\varphi_h \in V_h$;}
\begin{align*}
\displaystyle
| \varphi_h |_{p,V_h} &:= \left( | \varphi_h |_{W^{1,p}(\mathbb{T}_h)}^p + | \varphi_h |_{p,J}^p \right)^{\frac{1}{p}} \text{ with }  | \varphi_h |_{p,J} := \left( \sum_{F \in \mathcal{F}_h} \kappa_{p,F*} \| \Pi_F^{0} [\![ \varphi_h ]\!] \|_{L^p(F)}^p  \right)^{\frac{1}{p}},
\end{align*}
where for any $F \in \mathcal{F}_h$, we define the $L^2$-projection $\Pi_F^{0}: L^2(F) \to \mathbb{P}^{0}(F)$ as
\begin{align*}
\displaystyle
\int_F (\Pi_F^{0} \varphi - \varphi)   ds = 0 \quad \forall \varphi \in L^2(F).
\end{align*}

\begin{rem}
The definitions of $|\cdot|_{p,J}$ and $|\cdot|_{p,V_h}$ are applicable to any function $v \in W^{1,p}(\mathbb{T}_h)$, thereby establishing mesh-dependent seminorms on the broken Sobolev space $W^{1,p}(\mathbb{T}_h)$. In this study, we focus on utilising these seminorms within the finite element spaces $V_h \in \{V_h^{{DCCR}}, V_h^{{CR}}, V_{h0}^{{CR}}\}$.

In addition, within each of these discrete spaces, $|\cdot|_{p,V_h}$ serves as a norm. Specifically, if $|v_h|_{p,V_h} = 0$, it follows that $|v_h|_{W^{1,p}(\mathbb{T}_h)} = 0$, indicating that $v_h$ is piecewise constant over $\mathbb{T}_h$. For the case where $V_h = V_h^{{DCCR}}$, the condition $|v_h|_{p,J} = 0$ ensures that the constants on each element are uniform across adjacent elements and vanish on boundary faces, resulting in $v_h \equiv 0$. For $V_h = V_h^{{CR}}$ or $V_h = V_{h0}^{{CR}}$, the CR continuity conditions across interior faces, and in the latter case, the boundary moment conditions, necessitate that any piecewise constant function must be identically zero. Thus, $|\cdot|_{p,V_h}$ defines a norm on each of the spaces $V_h^{{DCCR}}, V_h^{{CR}}, V_{h0}^{{CR}}$.
\end{rem}

\color{black}
\subsection{Reference elements} \label{reference}
We first define the reference elements $\widehat{T} \subset \mathbb{R}^d$.

\subsubsection{Two-dimensional case} \label{reference2d}
Let $\widehat{T} \subset \mathbb{R}^2$ be a reference triangle with vertices $\hat{\bm p}_1 := (0,0){^{\top}}$, $\hat{\bm p}_2 := (1,0){^{\top}}$, and $\hat{\bm p}_3 := (0,1){^{\top}}$. 

\subsubsection{Three-dimensional case} \label{reference3d}
In the three-dimensional case, we consider the following two cases: (\roman{sone}) and (\roman{stwo}); see Condition \ref{cond2}.

Let $\widehat{T}_1$ and $\widehat{T}_2$ be reference tetrahedra with the following vertices:
\begin{description}
   \item[(\roman{sone})] $\widehat{T}_1$ has vertices $\hat{\bm p}_1 := (0,0,0){^{\top}}$, $\hat{\bm p}_2 := (1,0,0){^{\top}}$, $\hat{\bm p}_3 := (0,1,0){^{\top}}$, and $\hat{\bm p}_4 := (0,0,1){^{\top}}$;
 \item[(\roman{stwo})] $\widehat{T}_2$ has vertices $\hat{\bm p}_1 := (0,0,0){^{\top}}$, $\hat{\bm p}_2 := (1,0,0){^{\top}}$, $\hat{\bm p}_3 := (1,1,0){^{\top}}$, and $\hat{\bm p}_4 := (0,0,1){^{\top}}$.
\end{description}
Therefore, we set $\widehat{T} \in \{ \widehat{T}_1 , \widehat{T}_2 \}$. 
Note that the case (\roman{sone}) is called \textit{the regular vertex property}.

\subsection{Two-step affine mapping} \label{element=cond}
In anisotropic meshes, the mesh shape and element proportions are non-uniform, which directly affects the interpolation accuracy.  Existing interpolation error estimates typically assume an even or regular mesh, which may overestimate the error if applied directly to anisotropic meshes. To remedy this, we proposed {in \cite{Ish22,IshKobTsu21a,IshKobTsu23}} a new strategy on anisotropic meshes.

To an affine simplex $T \subset \mathbb{R}^d$, we construct two affine mappings $\Phi_{\widetilde{T}}: \widehat{T} \to \widetilde{T}$ and $\Phi_{T}: \widetilde{T} \to T$. First, we define the affine mapping $\Phi_{\widetilde{T}}: \widehat{T} \to \widetilde{T}$ as
\begin{align}
\displaystyle
\Phi_{\widetilde{T}}: \widehat{T} \ni \hat{\bm x} \mapsto \tilde{\bm x} := \Phi_{\widetilde{T}}(\hat{\bm x}) := {A}_{\widetilde{T}} \hat{\bm x} \in  \widetilde{T}, \label{aff=1}
\end{align}
where ${A}_{\widetilde{T}} \in \mathbb{R}^{d \times d}$ is an invertible matrix. We then define the affine mapping $\Phi_{T}: \widetilde{T} \to T$ as follows:
\begin{align}
\displaystyle
\Phi_{T}: \widetilde{T} \ni \tilde{\bm x} \mapsto \bm x := \Phi_{T}(\tilde{\bm x}) := {A}_{T} \tilde{\bm x} + \bm {b_{T}} \in T, \label{aff=2}
\end{align}
where $\bm {b_{T}} \in \mathbb{R}^d$ is a vector and ${A}_{T} \in O(d)$ denotes the rotation and mirror-imaging matrix. We define the affine mapping $\Phi: \widehat{T} \to T$ as
\begin{align*}
\displaystyle
\Phi := {\Phi}_{T} \circ {\Phi}_{\widetilde{T}}: \widehat{T} \ni \hat{\bm x} \mapsto \bm x := \Phi (\hat{\bm x}) =  ({\Phi}_{T} \circ {\Phi}_{\widetilde{T}})(\hat{\bm x}) = {A} \hat{\bm x} +\bm{ b_{T}} \in T, 
\end{align*}
where ${A} := {A}_{T} {A}_{\widetilde{T}} \in \mathbb{R}^{d \times d}$.

\subsubsection{Construct mapping $\Phi_{\widetilde{T}}: \widehat{T} \to \widetilde{T}$} \label{sec221} 
We consider the affine mapping \eqref{aff=1}. We define the matrix $ {A}_{\widetilde{T}} \in \mathbb{R}^{d \times d}$ as follows. We first define the diagonal matrix as
\begin{align}
\displaystyle
\widehat{A} :=  \diag (h_1,\ldots,h_d), \quad h_i \in \mathbb{R}_+ \quad \forall i. \label{aff=3}
\end{align}

For $d=2$, we define the regular matrix $\widetilde{A} \in \mathbb{R}^{2 \times 2}$ as
\begin{align}
\displaystyle
\widetilde{A} :=
\begin{pmatrix}
1 & s \\
0 & t \\
\end{pmatrix}, \label{aff=4}
\end{align}
with the parameters
\begin{align*}
\displaystyle
s^2 + t^2 = 1, \quad t \> 0.
\end{align*}
For the reference element $\widehat{T}$, let $\mathfrak{T}^{(2)}$ be a family of triangles.
\begin{align*}
\displaystyle
\widetilde{T} &= \Phi_{\widetilde{T}}(\widehat{T}) = {A}_{\widetilde{T}} (\widehat{T}), \quad {A}_{\widetilde{T}} := \widetilde {A} \widehat{A}
\end{align*}
with the vertices $\tilde{\bm p}_1 := (0,0)^{\top}$, $\tilde{\bm p}_2 := (h_1,0)^{\top}$ and $\tilde{\bm p}_3 :=(h_2 s , h_2 t)^{\top}$. Then, $h_1 = |\tilde{\bm p}_1 - \tilde{\bm p}_2| \> 0$ and $h_2 = |\tilde{\bm p}_1 - \tilde{\bm p}_3| \> 0$. 

For $d=3$, we define the regular matrices $\widetilde{A}_1, \widetilde{A}_2 \in \mathbb{R}^{3 \times 3}$ as follows:
\begin{align}
\displaystyle
\widetilde{A}_1 :=
\begin{pmatrix}
1 & s_1 & s_{21} \\
0 & t_1  & s_{22}\\
0 & 0  & t_2\\
\end{pmatrix}, \
\widetilde{A}_2 :=
\begin{pmatrix}
1 & - s_1 & s_{21} \\
0 & t_1  & s_{22}\\
0 & 0  & t_2\\
\end{pmatrix} \label{aff=5}
\end{align}
with the parameters
\begin{align*}
\displaystyle
\begin{cases}
s_1^2 + t_1^2 = 1, \ s_1 \> 0, \ t_1 \> 0, \ h_2 s_1 \leq h_1 / 2, \\
s_{21}^2 + s_{22}^2 + t_2^2 = 1, \ t_2 \> 0, \ h_3 s_{21} \leq h_1 / 2.
\end{cases}
\end{align*}
Therefore, we set $\widetilde{A} \in \{ \widetilde{A}_1 , \widetilde{A}_2 \}$. For the reference elements $\widehat{T}_i$, $i=1,2$, let $\mathfrak{T}_i^{(3)}$, $i=1,2$, be a family of tetrahedra.
\begin{align*}
\displaystyle
\widetilde{T}_i &= \Phi_{\widetilde{T}_i} (\widehat{T}_i) =  {A}_{\widetilde{T}_i} (\widehat{T}_i), \quad {A}_{\widetilde{T}_i} := \widetilde {A}_i \widehat{A}, \quad i=1,2,
\end{align*}
with the vertices
\begin{align*}
\displaystyle
&\tilde{\bm p}_1 := (0,0,0)^{\top}, \ \tilde{\bm p}_2 := (h_1,0,0)^{\top}, \ \tilde{\bm p}_4 := (h_3 s_{21}, h_3 s_{22}, h_3 t_2)^{\top}, \\
&\begin{cases}
\tilde{\bm p}_3 := (h_2 s_1 , h_2 t_1 , 0)^{\top} \quad \text{for case (\roman{sone})}, \\
\tilde{\bm p}_3 := (h_1 - h_2 s_1, h_2 t_1,0)^{\top} \quad \text{for case (\roman{stwo})}.
\end{cases}
\end{align*}
Subsequently, $h_1 = |\tilde{\bm p}_1 - \tilde{\bm p}_2| \> 0$, $h_3 = |\tilde{\bm p}_1 - \tilde{\bm p}_4| \> 0$, and
\begin{align*}
\displaystyle
h_2 =
\begin{cases}
|\tilde{\bm p}_1 - \tilde{\bm p}_3| \> 0  \quad \text{for case (\roman{sone})}, \\
|\tilde{\bm p}_2 - \tilde{\bm p}_3| \> 0  \quad \text{for case (\roman{stwo})}.
\end{cases}
\end{align*}

\subsubsection{Construct mapping $\Phi_{T}: \widetilde{T} \to T$}  \label{sec322}
We determine the affine mapping \eqref{aff=2} as follows. Let ${T} \in \mathbb{T}_h$ have vertices ${p}_i$ ($i=1,\ldots,d+1$). Let $b_{T} \in \mathbb{R}^d$ be the vector and ${A}_{T} \in O(d)$ be the rotation and mirror imaging matrix such that
\begin{align*}
\displaystyle
\bm p_{i} = \Phi_T (\tilde{\bm p}_i) = {A}_{T} \tilde{\bm p}_i + \bm{b_T}, \quad i \in \{1, \ldots,d+1 \},
\end{align*}
where vertices $\bm p_{i}$ ($i=1,\ldots,d+1$) satisfy the following conditions:

\color{black}
\begin{Cond}[Case in which $d=2$] \label{cond1}
Let ${T} \in \mathbb{T}_h$ have vertices ${\bm p_i}$ ($i=1,\ldots,3$). We assume that $\overline{{{\bm p}_2 {\bm p}_3}}$ is the longest edge of ${T}$, that is, $ h_{{T}} := |{{\bm p}_2 - {\bm p}_ 3}|$. We set $h_1 = |{{\bm p}_1 - {\bm p}_2}|$ and $h_2 = |{{\bm p}_1 - {\bm p}_3}|$. We then assume that $h_2 \leq h_1$. {Because $\frac{1}{2} h_T < h_1 \leq h_T$, ${h_1 \approx h_T}$.} 
\end{Cond}

\begin{Cond}[Case in which $d=3$] \label{cond2}
Let ${T} \in \mathbb{T}_h$ have vertices ${{\bm p}}_i$ ($i=1,\ldots,4$). Let ${L}_i$ ($1 \leq i \leq 6$) be the edges of ${T}$. We denote by ${L}_{\min}$  the edge of ${T}$ with the minimum length; that is, $|{L}_{\min}| = \min_{1 \leq i \leq 6} |{L}_i|$. We set $h_2 := |{L}_{\min}|$ and assume that 
\begin{align*}
\displaystyle
&\text{the endpoints of ${L}_{\min}$ are either $\{ {{\bm p}_1 , {\bm p}_3} \}$ or $\{ {{\bm p}_2 , {\bm p}_3}\}$}.
\end{align*}
Among the four edges sharing an endpoint with ${L}_{\min}$, we consider the longest edge ${L}^{({\min})}_{\max}$. Let ${{\bm p}}_1$ and ${\bm p}_2$ be the endpoints of edge ${L}^{({\min})}_{\max}$. Thus, we have
\begin{align*}
\displaystyle
h_1 = |{L}^{(\min)}_{\max}| = | {{\bm p}_1 - {\bm p}_2}|.
\end{align*}
We consider cutting $\mathbb{R}^3$ with a plane that contains the midpoint of the edge ${L}^{(\min)}_{\max}$ and is perpendicular to the vector ${{\bm p}_1 - {\bm p}_2}$. Thus, there are two cases. 
\begin{description}
  \item[(Type \roman{sone})] ${\bm p}_3$ and ${\bm p}_4$  belong to the same half-space;
  \item[(Type \roman{stwo})] ${\bm p}_3$ and ${\bm p}_4$  belong to different half-spaces.
\end{description}
In each case, we set
\begin{description}
  \item[(Type \roman{sone})] ${\bm p}_1$ and ${\bm p}_3$ as the endpoints of ${L}_{\min}$, that is, $h_2 =  |{\bm p}_1 - {\bm p}_3| $;
  \item[(Type \roman{stwo})] ${\bm p}_2$ and ${\bm p}_3$ as the endpoints of ${L}_{\min}$, that is, $h_2 =  |{\bm p}_2 - {\bm p}_3| $.
\end{description}
Finally, we set $h_3 = |{\bm p}_1 - {\bm p}_4|$. We implicitly assume that ${\bm p}_1$ and ${\bm p}_4$ belong to the same half-space. Additionally, note that ${h_1 \approx h_T}$.
\end{Cond}

\begin{lem} \label{lem351}
It holds that
\begin{subequations} \label{CN331}
\begin{align}
\displaystyle
\| \widehat{{A}} \|_2 &\leq  h_{T}, \quad \| \widehat{{A}} \|_2 \| \widehat{{A}}^{-1} \|_2 = \frac{\max \{h_1 , \cdots, h_d \}}{\min \{h_1 , \cdots, h_d \}}, \label{CN331a} \\
\| \widetilde{{A}} \|_2 &\leq 
\begin{cases}
\sqrt{2} \quad \text{if $d=2$}, \\
2  \quad \text{if $d=3$},
\end{cases}
\quad \| \widetilde{{A}} \|_2 \| \widetilde{{A}}^{-1} \|_2 \leq
\begin{cases}
\frac{h_1 h_2}{|T|_2} = \frac{H_{T}}{h_{T}} \quad \text{if $d=2$}, \\
\frac{2}{3} \frac{h_1 h_2 h_3}{|T|_3} = \frac{2}{3} \frac{H_{T}}{h_{T}} \quad \text{if $d=3$},
\end{cases}
\label{CN331b} \\
\| {A}_{T} \|_2 &= 1, \quad \| {A}_{T}^{-1} \|_2 = 1. \label{CN331c}
\end{align}
\end{subequations}
where a parameter $H_{T}$ is defined in Definition \ref{defi1}. Furthermore, we have
\begin{align}
\displaystyle
| \det ({A}_{\widetilde{T}}) | = | \det(\widetilde{{A}}) | | \det (\widehat{{A}}) | = \frac{|T|_d}{|\widetilde{T}|_d} \frac{|\widetilde{T}|_d}{|\widehat{T}|_d} = d ! |T|_d, \quad | \det ({A}_{T}) | = 1, \label{CN332}
\end{align}
where $\| {A} \|_2$ denotes the operator norm of ${A}$.
\end{lem}

\begin{pf*}
A proof can be found in \cite[Lemma 2]{IshKobTsu23}.	
\end{pf*}

\subsection{Two-step Piola transforms}
We adopt the following two-step Piola transformations.

\begin{defi}[Two-step Piola transforms] \label{piola=defi}
Let  $V(\widehat{T}) := \mathcal{C}(\widehat{T})^d$. The Piola transformation $\Psi := {\Psi}_{\widetilde{T}} \circ {\Psi}_{\widehat{T}} : V(\widehat{T}) \to V(T)$ is defined as
\begin{align}
\displaystyle
\Psi : V(\widehat{T})  &\to V(T) \label{piola} \\
\hat{\bm v} &\mapsto \bm v(\bm x) :=  {\Psi}(\hat{\bm v})(\bm x) = \frac{1}{\det({A})} {A} \hat{\bm v}(\hat{\bm x}), \notag
\end{align}
with	two Piola transformations:
\begin{align*}
\displaystyle
{\Psi}_{\widehat{T}}: V(\widehat{T}) &\to V(\widetilde{T})\\
\hat{\bm v} &\mapsto \tilde{\bm v}(\tilde{\bm x}) := {\Psi}_{\widehat{T}}(\hat{\bm v})(\tilde{\bm x}) := \frac{1}{\det ({A}_{\widetilde{T}})} {A}_{\widetilde{T}}\hat{\bm v}(\hat{\bm x}), \\
\Psi_{\widetilde{T}} :  V(\widetilde{T}) &\to V({T})  \\
\tilde{\bm v} &\mapsto \bm v(\bm x) :=  {\Psi}_{\widetilde{T}} (\tilde{\bm v})(\bm x) :=  \frac{1}{\det ({{A}}_{T}) } {A}_{T} \tilde{\bm v} (\tilde{\bm x}).
\end{align*}
\end{defi}

\subsection{Additional notations} \label{addinot}
We define the vectors ${\bm{r}}_n \in \mathbb{R}^d$, $n=1,\ldots,d$ as follows: If $d=2$,
{
\begin{align*}
\displaystyle
{\bm r}_1 := \frac{{\bm p}_2 - {\bm p}_1}{|{\bm p}_2 - {\bm p}_1|}, \quad {\bm r}_2 := \frac{{\bm p}_3 - {\bm p}_1}{|{\bm p}_3 - {\bm p}_1|},
\end{align*}
see Fig. \ref{affine_2d}, and if $d=3$,
\begin{align*}
\displaystyle
&{\bm r}_1 := \frac{{\bm p}_2 - {\bm p}_1}{|{ \bm p}_2 - {\bm p}_1|}, \quad {\bm r}_3 := \frac{{\bm p}_4 - {\bm p}_1}{|{\bm p}_4 - {\bm p}_1|}, \quad
\begin{cases}
\displaystyle
{\bm r}_2 := \frac{{\bm p}_3 - {\bm p}_1}{|{\bm p}_3 - {\bm p}_1|}, \quad \text{for (Type \roman{sone})}, \\
\displaystyle
{\bm r}_2 := \frac{{\bm p}_3 - {\bm p}_2}{|{\bm p}_3 - {\bm p}_2|} \quad \text{for (Type \roman{stwo})},
\end{cases}
\end{align*}
}
see Fig \ref{affine_3d_1} for (Type \roman{sone}) and Fig \ref{affine_3d_2} for (Type \roman{stwo}). 

\begin{figure}[htbp]
  \includegraphics[keepaspectratio, scale=0.45]{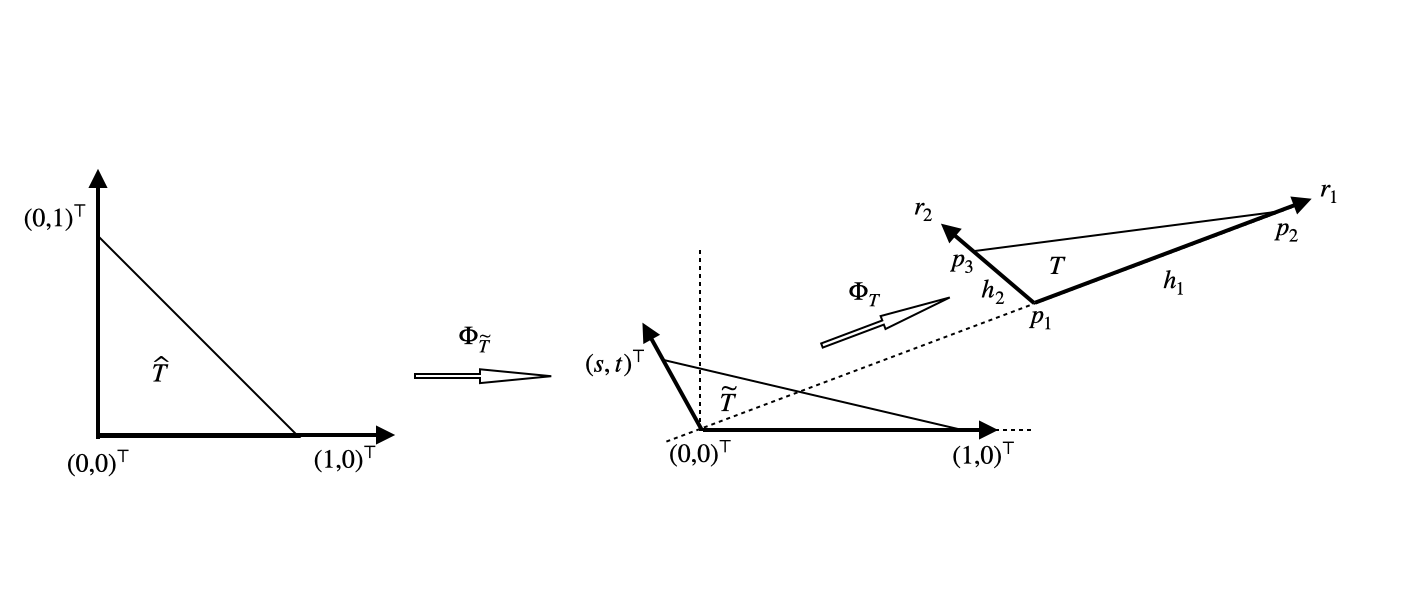}
\caption{Two-step affine mapping and vectors $r_i$, $i=1,2$}
\label{affine_2d}
\end{figure}

\begin{figure}[htbp]
  \begin{minipage}[b]{0.45\linewidth}
    \centering
    \includegraphics[keepaspectratio, scale=0.45]{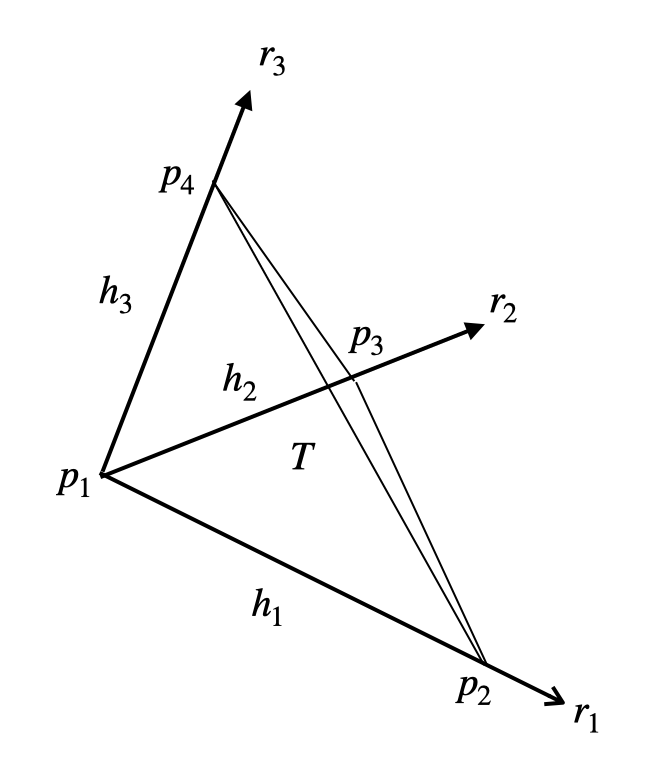}
    \caption{(Type \roman{sone}) Vectors $r_i$, $i=1,2,3$}
     \label{affine_3d_1}
  \end{minipage}
  \begin{minipage}[b]{0.45\linewidth}
    \centering
    \includegraphics[keepaspectratio, scale=0.45]{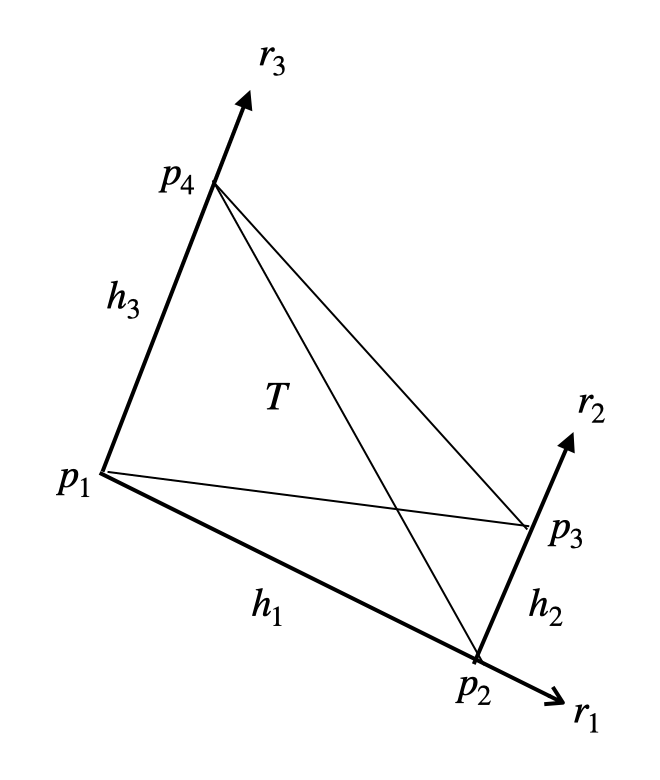}
    \caption{(Type \roman{stwo}) Vectors $r_i$, $i=1,2,3$}
     \label{affine_3d_2}
  \end{minipage}
\end{figure}

For a sufficiently smooth function $\varphi$ and a vector function ${\bm v} := (v_{1},\ldots,v_{d})^{\top}$, we define the directional derivative for $i \in \{ 1, \ldots,d \}$ as
{
\begin{align*}
\displaystyle
\frac{\partial \varphi}{\partial {{\bm r}_i}} &:= ( {\bm r}_i \cdot  {\bm \nabla}_{x} ) \varphi = \sum_{i_0=1}^d ({\bm r}_i)_{i_0} \frac{\partial \varphi}{\partial x_{i_0}^{}}, \\
\frac{\partial {\bm v}}{\partial {\bm r}_i} &:= \left(\frac{\partial v_{1}}{\partial {\bm r}_i}, \ldots, \frac{\partial v_{d}}{\partial {\bm r}_i} \right)^{\top} 
= ( ({\bm r}_i  \cdot {\bm \nabla}_{x}) v_{1}, \ldots, ({\bm r}_i  \cdot {\bm \nabla}_{x} ) v_{d} )^{\top}.
\end{align*}
}
For a multiindex $\beta = (\beta_1,\ldots,\beta_d) \in \mathbb{N}_0^d$, we use the notation
{
\begin{align}
\displaystyle
\partial^{\beta} \varphi := \frac{\partial^{|\beta|} \varphi}{\partial x_1^{\beta_1} \ldots \partial x_d^{\beta_d}}, \quad \partial^{\beta}_{r} \varphi := \frac{\partial^{|\beta|} \varphi}{\partial {\bm r}_1^{\beta_1} \ldots \partial {\bm r}_d^{\beta_d}}, \quad h^{\beta} :=  h_{1}^{\beta_1} \cdots h_{d}^{\beta_d}. \label{partial=sym}
\end{align}
}
We note that $\partial^{\beta} \varphi \neq  \partial^{\beta}_{r} \varphi$.

\subsection{$L^2$-orthogonal projection} \label{sec=L2ortho}
This section considers error estimates of the $L^2$-orthogonal projection, e.g., for a standard argument, see \cite[Section 11.5.3]{ErnGue21a}. Here, the discussion is based on the two-step affine mapping.

 Let $\widehat{T} \subset \mathbb{R}^d$ be the reference element defined in Section \ref{reference}. The $L^2$-orthogonal projection onto $\widehat{P} := \mathbb{P}^0(\widehat{T})$ is the linear operator ${\Pi}_{\widehat{T}}^{0}: L^1(\widehat{T}) \to \widehat{P}$ defined as
\begin{align}
\displaystyle
\int_{\widehat{T}} ({\Pi}_{\widehat{T}}^{0} \hat{\varphi} - \hat{\varphi}) d \hat{x} = 0 \quad \forall \hat{\varphi} \in L^1(\widehat{T}). \label{L2ortho=ref}
\end{align}
Because ${\Pi}_{\widehat{T}}^{0} \hat{\varphi} - \hat{\varphi}$ and ${\Pi}_{\widehat{T}}^{0} \hat{\varphi} - \hat{q}$ are $L^2$-orthogonal for any $\hat{q} \in \widehat{P}$, the Pythagorean identity yields 
\begin{align*}
\displaystyle
\| \hat{\varphi} - \hat{q} \|^2_{L^2(\widehat{T})}
= \| \hat{\varphi} - {\Pi}_{\widehat{T}}^{0} \hat{\varphi} \|^2_{L^2(\widehat{T})} + \| {\Pi}_{\widehat{T}}^{0} \hat{\varphi} - \hat{q} \|^2_{L^2(\widehat{T})}.
\end{align*}
This implies that
\begin{align*}
\displaystyle
\Pi_{\widehat{T}}^{0} \hat{\varphi} = \argmin_{\hat{q} \in \widehat{P}} \| \hat{\varphi} - \hat{q} \|_{L^2(\widehat{T})}.
\end{align*}
Therefore, $\widehat{P}$ is pointwise invariant under $\Pi_{\widehat{T}}^{0}$. Let $\Phi_{\widetilde{T}}: \widehat{T} \to \widetilde{T}$ and $\Phi_{T}: \widetilde{T} \to T$ be the two affine mappings defined in Section \ref{element=cond}. For any ${T} \in \mathbb{T}_h$ with $\widetilde{T} = {\Phi}_{\widetilde{T}}(\widehat{T})$ and $T ={\Phi}_{T} (\widetilde{T})$, let $\hat{\varphi} :=  \tilde{\varphi} \circ {\Phi_{\widetilde{T}}}$ and $ \tilde{\varphi} := \varphi \circ {\Phi_T}$. Furthermore, we set
\begin{align*}
\displaystyle
 \widetilde{P} &:= \{ \hat{q} \circ \Phi_{\widetilde{T}}^{-1} ; \ \hat{q} \in \widehat{{P}}\}, \\
 {P} &:= \{ \tilde{q} \circ \Phi_T ; \ \tilde{q} \in \widetilde{{P}}\}.
\end{align*}
The $L^2$-orthogonal projections onto $\widehat{P}$ and $P$ are respectively the linear operators $\Pi^0_{\widetilde{T}}: L^1(\widetilde{T}) \to \widetilde{P}$ and $\Pi^0_{{T}}: L^1({T}) \to {P}$ defined as
\begin{align*}
\displaystyle
&\int_{\widetilde{T}} ({\Pi}_{\widetilde{T}}^{0} \tilde{\varphi} - \tilde{\varphi}) d \tilde{x} = 0 \quad \forall \tilde{\varphi} \in L^1(\widetilde{T}), \\
&\int_{{T}} ({\Pi}_{{T}}^{0} {\varphi} - {\varphi}) d {x} = 0 \quad \forall {\varphi} \in L^1({T}).
\end{align*}
Then,  $\widetilde{P}$ and $P$ are respectively pointwise invariant under $\Pi_{\widetilde{T}}^{0}$ and $\Pi_T^0$.

We also define the global interpolation $\Pi_h^0$ to space $V_{h}^{DC(0)}$ as
\begin{align*}
\displaystyle
(\Pi_h^0 \varphi)|_{T} := \Pi_{{T}}^0 (\varphi|_{T}) \quad \forall T \in \mathbb{T}_h, \quad \forall \varphi \in L^1(\Omega).
\end{align*}

\begin{lem}
Let $q \in [1,\infty)$. It holds that
\begin{align}
\displaystyle
{
\| \Pi_{\widehat{T}}^0 \hat{\varphi} \|_{L^q(\widehat{T})} 
\leq \| \hat{\varphi} \|_{L^q(\widehat{T})}} \quad \forall \hat{\varphi} \in L^{q}(\widehat{T}). \label{chapp821}
\end{align}
\end{lem}

\begin{pf*}
Let $\hat{\varphi} \in L^{q}(\widehat{T})$. From the definition of the $L^2$-orthogonal projection \eqref{L2ortho=ref},
\begin{align*}
\displaystyle
\Pi_{\widehat{T}}^0 \hat{\varphi}
= \frac{1}{|\widehat{T}|_d} \int_{\widehat{T}} \hat{\varphi} d \hat{x} \in \mathbb{P}^0(\widehat{T}).
\end{align*}
Using the H\"older's inequality with $\frac{1}{q} + \frac{1}{q^{\prime}} = 1$ yields
\begin{align*}
\displaystyle
| \Pi_{\widehat{T}}^0 \hat{\varphi} |
= \frac{1}{|\widehat{T}|_d} \left|  \int_{\widehat{T}} \hat{\varphi} d \hat{x} \right|
\leq  \frac{1}{|\widehat{T}|_d} \| \hat{\varphi} \|_{L^q(\widehat{T})} \| 1 \|_{L^{q^{\prime}}(\widehat{T})} = \frac{1}{|\widehat{T}|_d} \| \hat{\varphi} \|_{L^q(\widehat{T})} |\widehat{T}|_d^{\frac{1}{q^{\prime}}} = |\widehat{T}|_d^{\frac{1}{q^{\prime}}-1} \| \hat{\varphi} \|_{L^q(\widehat{T})},
\end{align*}
which leads to
\begin{align*}
\displaystyle
\| \Pi_{\widehat{T}}^0 \hat{\varphi} \|_{L^q(\widehat{T})} 
&= \left( \int_{\widehat{T}} | \Pi_{\widehat{T}}^0 \hat{\varphi} |^q d \hat{x} \right)^{\frac{1}{q}} =  | \Pi_{\widehat{T}}^0 \hat{\varphi} |  |\widehat{T}|_d^{\frac{1}{q}} \\
&\leq  |\widehat{T}|_d^{\frac{1}{q} + \frac{1}{q^{\prime}}-1} \| \hat{\varphi} \|_{L^q(\widehat{T})} = \| \hat{\varphi} \|_{L^q(\widehat{T})}.
\end{align*}
\end{pf*}

\color{black}
The following theorem gives an anisotropic error estimate of the projection $\Pi_T^0$.

\begin{thr} \label{thr822}
Let $p \in [1,\infty)$ and $q \in [1,\infty)$ be such that
\begin{align}
\displaystyle
W^{1,p}({T}) \hookrightarrow L^q({T}), \label{Sobolev511}
\end{align}
that is $1 - \frac{d}{p} \geq - \frac{d}{q}$. It then holds that, for any $\hat{\varphi} \in W^{1,p}(\widehat{T})$ with ${\varphi} := \hat{\varphi} \circ {\Phi}^{-1}$,
\begin{align}
\displaystyle
\| \Pi_{T}^0 \varphi - \varphi \|_{L^q(T)} \leq {C^{L^2}(d,p,q)} |T|_d^{\frac{1}{q} - \frac{1}{p}} \sum_{i=1}^d h_i \left\| \frac{\partial \varphi}{\partial \bm r_i} \right\|_{L^p(T)}, \label{chap822}
\end{align}
{where $C^{L^2}(d,p,q)$ depends only on the dimension $d$, the exponents $p$ and $q$ and is independent of the physical element $T$, the mesh size $h$, and the semi-regular parameter.
}
\end{thr}

\begin{pf*}
Let $p\in[1,\infty)$ and $q\in[1,\infty)$ be such that \eqref{Sobolev511} holds. Let ${\varphi} = \hat{\varphi} \circ {\Phi}^{-1}$.

By the invariance of $P$ under the $L^2$-projection $\Pi_{T}^0$ (see Section \ref{sec=L2ortho}) and by the definition of the affine mapping  $\Phi: \widehat{T} \to T$, we have
\begin{align*}
\displaystyle
\Pi_{T}^0 \varphi - \varphi
=  (\Pi_{\widehat{T}}^0 \hat{\varphi} - \hat{\varphi}) \circ {\Phi}^{-1},
\end{align*}
so that a change of variables $x = \Phi(\hat x)$ yields
\begin{align}
\displaystyle
\| \Pi_{T}^0 \varphi - \varphi \|_{L^q(T)}
&= |\det({A}_{\widetilde{T}})|^{\frac{1}{q}} \| \Pi_{\widehat{T}}^0 \hat{\varphi} - \hat{\varphi} \|_{L^q(\widehat{T})}. \label{chap824}
\end{align}
where we used $|\det (A_T)| = 1$. We set
\begin{align*}
\displaystyle
\hat{\eta} := \frac{1}{|\widehat{T}|_d} \int_{\widehat{T}} \hat{\varphi} d \hat{x} \in \mathbb{P}^0(\widehat{T}).
\end{align*}
From the triangle inequality and $\Pi_{\widehat{T}}^0 \hat{\eta} = \hat{\eta}$, we have
\begin{align}
\displaystyle
\| \Pi_{\widehat{T}}^0 \hat{\varphi} - \hat{\varphi} \|_{L^q(\widehat{T})}
&\leq \| \Pi_{\widehat{T}}^0  (\hat{\varphi} - \hat{\eta}) \|_{L^q(\widehat{T})}  + \| \hat{\eta} - \hat{\varphi} \|_{L^q(\widehat{T})}. \label{chap825}
\end{align}
Using \eqref{chapp821} for the first term on the right-hand side of \eqref{chap825}, we have
\begin{align}
\displaystyle
\| \Pi_{\widehat{T}}^0  (\hat{\varphi} - \hat{\eta}) \|_{L^q(\widehat{T})} 
\leq \| \hat{\varphi} - \hat{\eta} \|_{L^q(\widehat{T})}. \label{chap826}
\end{align}
For the second term on the right-hand side of \eqref{chap825} and \eqref{chap826}, by the Sobolev embedding theorem on the fixed reference simplex $\widehat{T}$, there exists a constant $C^{Sob} = C^{Sob}(d,p,q) > 0$ such that
\begin{align}
\displaystyle
\| \hat{\varphi} - \hat{\eta} \|_{L^q(\widehat{T})} \leq C^{Sob} \| \hat{\varphi} - \hat{\eta} \|_{W^{1,p}(\widehat{T})}, \label{chap827}
\end{align}
where $C^{Sob}$ depends only on the dimension $d$, the exponents $p$ and $q$
(and on the shape of the fixed reference simplex $\widehat{T}$), and is independent of the physical element $T$, the mesh size $h$, and the semi-regular parameter.

Combining \eqref{chap824}, \eqref{chap825}, \eqref{chap826}, and \eqref{chap827}, we have
\begin{align}
\displaystyle
\| \Pi_{T}^0 \varphi - \varphi \|_{L^q(T)}
&\leq 2 C^{Sob} |\det({A}_{\widetilde{T}})|^{\frac{1}{q}} \| \hat{\varphi} - \hat{\eta} \|_{W^{1,p}(\widehat{T})}. \label{chap828}
\end{align}
From {the Poincar\'e--Steklov} inequality (e.g., see \cite[Lemma 3.24]{ErnGue21a}),
\begin{align*}
\displaystyle
\| \hat{\varphi} - \hat{\eta} \|_{W^{1,p}(\widehat{T})}^p
&= | \hat{\varphi}  |_{W^{1,p}(\widehat{T})}^p + \| \hat{\varphi} - \hat{\eta} \|_{L^{p}(\widehat{T})}^p \\
&\leq  | \hat{\varphi}  |_{W^{1,p}(\widehat{T})}^p + C^{PS}(p) h_{\widehat{T}}^p  | \hat{\varphi}  |_{W^{1,p}(\widehat{T})}^p,
\end{align*}
which leads to
\begin{align}
\displaystyle
\| \hat{\varphi} - \hat{\eta} \|_{W^{1,p}(\widehat{T})}
&\leq C^{PS1}(d,p) | \hat{\varphi}  |_{W^{1,p}(\widehat{T})}, \label{chap829}
\end{align}
where we set
\begin{align*}
\displaystyle
C^{PS1}(d,p) := \left( 1 +  C^{PS}(p) h_{\widehat{T}}^p \right)^{\frac{1}{p}}.
\end{align*}
Because the reference simplex $\widehat{T}$ is fixed once and for all, $h_{\widehat{T}}$ is a dimension-dependent constant and $C^{PS1}(d,p)$ depends only on $d$ and $p$.

Using the new scaling argument \cite[Lemma 6]{IshKobTsu23} ($\ell = 1$ and $m=0$), we then have
\begin{align}
\displaystyle
 |\hat{\varphi}|_{W^{1,p}(\widehat{T})} 
&\leq C^{Ish}(d,p)  |\det({A}_{\widetilde{T}})|^{-\frac{1}{p}} \sum_{i=1}^d h_i \left\| \frac{\partial \varphi}{\partial \bm r_i} \right\|_{L^p(T)},  \label{chap8210}
\end{align}
where $ C^{Ish}(d,p)$ depends only on the exponent $d$ and the exponent $p$, and is independent of the physical element $T$, the mesh size $h$, and the semi-regular parameter.

Therefore, combining \eqref{chap828}, \eqref{chap829} and \eqref{chap8210} proves \eqref{chap822} with
\begin{align*}
\displaystyle
C^{L^2}(d,p,q) := 2 C^{Sob}(d,p,q) \cdot C^{PS}(d,p) \cdot C^{Ish}(d,p) \cdot (d !)^{\frac{1}{q} - \frac{1}{p}},
\end{align*}
which depends only on $d,p,q$ and is independent of the physical element $T$,
the mesh size $h$, and the semi-regular parameter. Here, we used
\begin{align*}
\displaystyle
| \det ({A}_{\widetilde{T}}) | = | \det(\widetilde{{A}}) | | \det (\widehat{{A}}) | = \frac{|T|_d}{|\widetilde{T}|_d} \frac{|\widetilde{T}|_d}{|\widehat{T}|_d} = d ! |T|_d.
\end{align*}
\end{pf*}

\color{black}
\section{Semi-regular geometric mesh condition}

\subsection{New geometric mesh condition}
We proposed a new geometric parameter $H_T$ in \cite{IshKobTsu21a,IshKobTsu23}. This parameter represents the flatness of a simplex.

\begin{defi} \label{defi1}
 {We define the parameter $H_{{T}}$ as}
{
\begin{align*}
\displaystyle
H_{{T}} := \frac{\prod_{i=1}^d h_i}{|{T}|_d} h_{{T}}.
\end{align*}
}
\end{defi}

The following geometric condition is equivalent to the maximum angle condition (\cite{IshKobSuzTsu21}).

\begin{assume} \label{neogeo=assume}
A family of meshes $\{ \mathbb{T}_h\}$ has a semi-regular property if there exists $\gamma_0 \> 0$ such that
\begin{align}
\displaystyle
\frac{H_{T}}{h_{T}} \leq \gamma_0 \quad \forall \mathbb{T}_h \in \{ \mathbb{T}_h \}, \quad \forall T \in \mathbb{T}_h. \label{NewGeo}
\end{align}
\end{assume}

The quantity $H_T/h_T$ can be easily calculated in the numerical process of finite element methods. Therefore, the new condition may be useful in the case of adaptive finite element methods. We expect the new mesh condition to become an alternative to the maximum-angle condition.

\subsection{RT finite element interpolation operator} \label{RTsp}
For $T \in \mathbb{T}_h$, the local RT polynomial space is defined as
\begin{align}
\displaystyle
\mathbb{RT}^0(T) := \mathbb{P}^0(T)^d + {\bm x} \mathbb{P}^0(T), \quad {\bm x} \in \mathbb{R}^d. \label{RTsp}
\end{align}
Let ${I}_{T}^{RT}: W^{1,1}(T)^d \to \mathbb{RT}^0(T)$ be the RT interpolation operator such that for any ${\bm v} \in W^{1,1}(T)^d$,
\begin{align}
\displaystyle
{I}_{T}^{RT}: W^{1,1}(T)^d \ni {\bm v} \mapsto {I}_{T}^{RT} {\bm v} := \sum_{i=1}^{d+1} \left(  \int_{{F}_{T,i}} {\bm v} \cdot {\bm n}_{T,i} d{s} \right) {\bm \theta}_{T,i}^{RT} \in \mathbb{RT}^0(T), \label{RTint}
\end{align}
where ${\bm \theta}_{T,i}^{RT}$ is the local shape {basis} function (e.g. \cite[p. 162]{ErnGue21a}) and ${\bm n}_{T,i}$ is a fixed unit normal to ${F}_{T,i}$. 

The RT finite-element space is defined as follows:
\begin{align*}
\displaystyle
V^{RT}_{h} &:= \{ {\bm v_h} \in L^1(\Omega)^d: \  {\bm v_h}|_T \in \mathbb{RT}^0(T), \ \forall T \in \mathbb{T}_h, \  [\![ {\bm v_h} \cdot {\bm n} ]\!]_F = 0, \ \forall F \in \mathcal{F}_h^i \}.
\end{align*}
We define the following global RT interpolation ${I}_{h}^{RT} : W^{1,1}(\Omega)^d \to V^{RT}_{h}$ as
\begin{align*}
\displaystyle
({I}_{h}^{RT} {\bm v} )|_{T} := {I}_{T}^{RT} ({\bm v}|_{T}) \quad \forall T \in \mathbb{T}_h, \quad \forall {\bm v} \in W^{1,1}(\Omega)^d.
\end{align*}
The following two lemmata are divided into the element on $\mathfrak{T}^{(2)}$ or $\mathfrak{T}_1^{(3)}$ and the element on $\mathfrak{T}_2^{(3)}$ in Section \ref{element=cond}. 

\begin{lem} \label{lem1151}
Let $p \in [1,\infty)$. Let $T\in \mathbb{T}_h$ satisfy Condition \ref{cond1} or Condition \ref{cond2} with $T ={\Phi}_{T} (\widetilde{T})$ and $\widetilde{T} = {\Phi}_{\widetilde{T}}(\widehat{T})$, where $\widetilde{T} \in \mathfrak{T}^{(2)}$ or $\widetilde{T} \in \mathfrak{T}_1^{(3)}$. Then, for any $\bm \hat{v} \in W^{1,p}(\widehat{T})^d$ with $\bm \tilde{v}= {\Psi}_{\widehat{T}} \bm \hat{v}$ and $\bm {v} = {\Psi}_{\widetilde{T}} \bm \tilde{v}$,  
\begin{align}
\displaystyle
\| {I_{T}^{RT} \bm v} \|_{L^p({T})^d} 
&\leq {C^{LRTsta1}(d,p)} \left[ \frac{H_{T}}{h_{T}} \left( \| \bm v \|_{L^p(T)^d} + \sum_{|\varepsilon|=1} h^{\varepsilon} \left \| \partial_{\bm r}^{\varepsilon} \bm v \right \|_{L^p(T)^d} \right ) + h_{T} \| \bm \nabla \cdot \bm {v} \|_{L^p({T})} \right]. \label{RT51}
\end{align}
{Here, the constant $C^{LRTsta1}(d,p)$ depends only on $d$ and $p$ and is independent of $h$ and of the aspect ratios and interior angles of the simplices.}
\end{lem}

\begin{pf*}
A proof can be found in \cite[Lemma 8]{Ish22}.
\end{pf*}

\begin{lem} \label{lem1152}
Let $p \in [1,\infty)$ and $d=3$. Let $T\in \mathbb{T}_h$ satisfy Condition \ref{cond2} with $T ={\Phi}_{T} (\widetilde{T})$ and $\widetilde{T} = {\Phi}_{\widetilde{T}}(\widehat{T})$, where $\widetilde{T} \in \mathfrak{T}_2^{(3)}$. Then,  for any $\bm \hat{v} \in W^{1,p}(\widehat{T})^3$ with $\bm \tilde{v}= {\Psi}_{\widehat{T}} \bm \hat{v}$ and $\bm {v} = {\Psi}_{\widetilde{T}} \bm \tilde{v}$,  
\begin{align}
\displaystyle
\| {I_{T}^{RT} \bm v} \|_{L^p({T})^3} 
&\leq {C^{LRTsta2}(d,p)} \frac{H_{T}}{h_{T}} \left[  \| \bm v \|_{L^p(T)^3}  + h_T \sum_{k=1}^3 \left \| \frac{\partial \bm v}{\partial \bm{r}_k} \right \|_{L^p(T)^3} \right]. \label{RT58}
\end{align}
{Here, the constant $C^{LRTsta2}(d,p)$ depends only on $d$ and $p$ and is independent of $h$ and of the aspect ratios and interior angles of the simplices.}
\end{lem}

\begin{pf*}
A proof can be found in \cite[Lemma 8]{Ish22}.
\end{pf*}

Lemmata \ref{lem1151} and \ref{lem1152} yield the following corollary.

\begin{cor}[Stability] \label{RT=sta}
Let $p \in [1,\infty)$.  {We impose Assumption \ref{neogeo=assume}.} Then,
\begin{align*}
\displaystyle
\| {I_{h}^{RT} \bm v} \|_{L^p({\Omega})^d}
&\leq {C^{RTsta}(d,p,\gamma_0)} \| \bm v\|_{W^{1,p}(\Omega)^d} \quad \forall \bm v \in W^{1,p}(\Omega)^d. 
\end{align*}
{Here, the constant $C^{RTsta}(d,p,\gamma_0)$ depends only on $d$, $p$ and the semi-regular mesh parameter $\gamma_0$, and is independent of $h$ and of the aspect ratios and interior angles of the simplices.}
\end{cor}

The following two theorems are based on an element $T$ satisfying  Type \roman{sone} or Type \roman{stwo} in Section \ref{element=cond} when $d=3$.

\begin{thr} \label{thr3}
Let $p \in [1,\infty)$. Let $T$ with $T ={\Phi}_{T} (\widetilde{T})$ and $\widetilde{T} = {\Phi}_{\widetilde{T}}(\widehat{T})$ be an element with Conditions \ref{cond1} or \ref{cond2} satisfying (Type \roman{sone}) in Section \ref{element=cond} when $d=3$. Let $\{ {T} , \mathbb{RT}^0({T}) , {\Sigma} \}$ be the RT finite element and $I_{T}^{RT}$ the local interpolation operator defined in \eqref{RTint}.Then, for any $\hat{\bm v} \in W^{1,p}(\widehat{T})^d$ with $\tilde{\bm v}= {\Psi}_{\widehat{T}} \hat{\bm v}$ and ${\bm v} = {\Psi}_{\widetilde{T}} \tilde{\bm v}$,  
\begin{align}
\displaystyle
\| {I}_{T}^{RT} \bm v - \bm v \|_{L^p(T)^d} 
&\leq  {C^{LRTint1}(d,p)} \left( \frac{H_{T}}{h_{T}} \sum_{i=1}^d h_i \left \|  \frac{\partial \bm v}{\partial r_i} \right \|_{L^p(T)^d} +  h_{T} \| \div \bm{v} \|_{L^{p}({T})} \right). \label{RT5}
\end{align}
{Here, the constant $C^{LRTint1}(d,p)$ depends only on $d$ and $p$ and is independent of $h$ and of the aspect ratios and interior angles of the simplices.}
\end{thr}

\begin{pf*}
A proof can be found in \cite[Theorem 2]{Ish22}.
\end{pf*}

\begin{thr} \label{thr4}
Let $p \in [1,\infty)$ and $d=3$.  Let $T$ with $T ={\Phi}_{T} (\widetilde{T})$ and $\widetilde{T} = {\Phi}_{\widetilde{T}}(\widehat{T})$ be an element with Condition \ref{cond2} satisfying (Type \roman{stwo}) in Section \ref{element=cond}.Let $\{ {T} , \mathbb{RT}^0({T}) , {\Sigma} \}$ be the RT finite element and $I_{T}^{RT}$ the local interpolation operator defined in \eqref{RTint}. Then, for any $\hat{\bm v} \in W^{1,p}(\widehat{T})^3$ with $\tilde{\bm v}= {\Psi}_{\widehat{T}} \hat{\bm v}$ and ${\bm v} = {\Psi}_{\widetilde{T}} \tilde{\bm v}$, 
\begin{align}
\displaystyle
&\| {I}_{T}^{RT} \bm v - \bm v \|_{L^p(T)^3} 
\leq {C^{LRTint2}(d,p)} \frac{H_{T}}{h_{T}} \Biggl(  h_{T} |\bm v|_{W^{1,p}(T)^3} \Biggr). \label{RT6}
\end{align}
{Here, the constant $C^{LRTint2}(d,p)$ depends only on $d$ and $p$ and is independent of $h$ and of the aspect ratios and interior angles of the simplices.}
\end{thr}

\begin{pf*}
A proof can be found in \cite[Theorem 3]{Ish22}.
\end{pf*}

\begin{cor} \label{coro=gloRT}
Let $p \in [1,\infty)$.  We impose Assumption \ref{neogeo=assume}. Then, for any $\hat{\bm v} \in W^{1,p}(\widehat{T})^d$ with $\tilde{\bm v}= {\Psi}_{\widehat{T}} \hat{\bm v}$ and ${\bm v} = {\Psi}_{\widetilde{T}} \tilde{\bm v}$, 
\begin{align}
\displaystyle
\| {I}_{T}^{RT} \bm v - \bm v \|_{L^p(T)^d} 
&\leq {C^{RTint}(d,p,\gamma_0)}  h_{T} |\bm v|_{W^{1,p}(T)^d}. \label{RTint=coro}
\end{align}
{Here, the constant $C^{{RTint}}(d,p,\gamma_0)$ depends only on $d$, $p$ and the semi-regular mesh parameter $\gamma_0$, and is independent of $h$ and of the aspect ratios and interior angles of the simplices.}
\end{cor}

\begin{rem}
In the classical shape-regular framework, the local Raviart-Thomas (RT) interpolation operator $I_T^{{RT}}$, as defined in \eqref{RTint}, can be extended by continuity from $W^{1,1}(T)^d$ to the broader space
\begin{align*}
\displaystyle
W(T) := \{ \bm {v} \in L^q(T)^d; \ \div \bm{v}\in L^2(T) \} = H(\div,T) \cap L^q(T)^d, \quad q>2,
\end{align*}
as demonstrated in \cite[Section~2.5.1]{BofBreFor13} and \cite[Section~17.2]{ErnGue21a}. Specifically, for $\bm{v}\in W(T)$, all the degrees of freedom in \eqref{RTint} are well-defined, and $I_T^{{RT}}\bm{v}\in \mathbb{RT}^0(T)$.

In this paper, we operate under the semi-regular anisotropic mesh condition and employ the local estimates \eqref{RT51} and \eqref{RT58} from \cite{Ish22}. These bounds are established within a $W^{1,p}$-framework and incorporate directional derivatives of $\bm{v}$ in addition to its divergence, along with the two-step affine mapping and anisotropic trace inequalities. Currently, we are unable to extend such anisotropic estimates to rougher data $\bm{v}\in W(T)$ (or even to $\bm{v} \in H(\div,T)$), particularly in terms of eliminating the gradient terms on the right-hand side of \eqref{RT51} while maintaining the same dependence of the constants on the semi-regularity parameter.

Consequently, throughout Section \ref{RTsp}, we retain the regularity assumption $\bm{v}\in W^{1,p}(T)^d$ in all stability and RT interpolation estimates, including  Lemmata \ref{lem1151}-\ref{lem1152}, Corollaries \ref{RT=sta}-\ref{coro=gloRT}, and Theorems \ref{thr3}-\ref{thr4}. It is also noteworthy that our discrete $L^q - L^p$ Sobolev inequalities in Section \ref{sec;proof=main} rely solely on the global stability estimate in Corollary \ref{RT=sta}.

We also refer to our previous work~\cite{Ish22} for comprehensive discussions on anisotropic RT interpolation on semi-regular meshes and for counterexamples illustrating that Babu\v{s}ka--Aziz type techniques and purely $H(\div)$-based approaches are not applicable in this context; see particularly in \cite[Section 4]{Ish22}.

Finally, we recall that our local RT interpolation estimates in Theorems \ref{thr3}-\ref{thr4} are based on the anisotropic results of \cite{Ish22}. In the Type~(i) configuration, \cite[Theorem~2]{Ish22} yields the genuinely anisotropic bound \eqref{RT5}, expressed in terms of directional derivatives and the divergence of $v$. In the Type~(ii) configuration, however, \cite[Theorem~3]{Ish22} currently provides only an estimate of the form \eqref{RT6}, involving the full $W^{1,p}(T)^3$-seminorm. We do not claim any refinement of this result here, and obtaining a sharper anisotropic estimate analogous to \eqref{RT5} for Type~(ii) elements under the same semi-regular mesh condition appears to require additional ideas and is left open.

\end{rem}

\begin{rem}[On the dependence of constants]
Throughout this paper, we denote the parameter dependence of generic constants by expressions such as $C^{LRTsta1}(d,p)$ or $C^{dS0}(d,\Omega,p,q,\gamma_0)$. Notably, these constants are independent of the mesh size $h$, as well as the element aspect ratios and interior angles. When applicable, their dependence on anisotropy is solely through the semi-regularity parameter $\gamma_0$.

In the context of anisotropic RT interpolation estimates, as exemplified in Lemma \ref{lem1151}, we document only the dependence on pertinent parameters, without explicitly tracking all intermediate constants. While it is theoretically feasible to derive a closed-form expression in terms of $d$ and $p$, such an endeavour would necessitate extensive bookkeeping through the two-step affine/Piola transformations and auxiliary inequalities, which exceeds the scope of this paper. We intend to address explicit constant tracking in future research.
\end{rem}

\color{black}

\section{Lemmata for analysis}

\begin{lem}[Bogovski\u{\i}-type lemma] \label{lem:Bogovskii}
Let $D \subset \mathbb{R}^d$, $d \in \{ 2 , 3 \}$ be a bounded, connected domain that is a finite union of open sets with Lipschitz-continuous boundaries. Then, there exists an operator $\mathcal{L}$ from $L^1_0(D)$ into $L^1(D)^d$ such that
\begin{align}
\displaystyle
\div (\mathcal{L} f) = f \quad \forall f \in L^1_0(D).
\end{align}
Furthermore, the operator $\mathcal{L}$ maps $\mathcal{D}(D) \cap L^1_0(D)$ into $\mathcal{D}(D) ^d$ and for each real $r \in (1,\infty)$, is continuous from $L^r_0(D)$ into $W^{1,r}_0(D)^d$; there exists a constant ${C^{Bog}(d,D,r)}$ such that
\begin{align}
\displaystyle
| \mathcal{L} f |_{W^{1,r}(D)^d} \leq {C^{Bog}(d,D,r)} \| f \|_{L^r(D)} \quad \forall f \in L^r_0(D), \label{ineq=4=2}
\end{align}
where $L^r_0(D) := \{ v \in L^r(D): \ \int_D v dx = 0 \}$. {Here, the constant $C^{Bog}(d,D,r)$ depends only on $r$, $d$, and the geometry of $D$, and is, in particular independent of the mesh $\mathbb{T}_h$.} {Furthermore, if $D$ is star-shaped with respect to a ball $B_{\rho} \subset D$ of radius $\rho$, then $C^{Bog}(d,D,r)$ admits an upper bound in terms of $d$, $r$, and the chunkiness parameter $\eta_D := \frac{\diam (D)}{\rho}$. For more general domains obtained by patching such subdomains (e.g. via a partition-of-unity argument), the dependence of the constant on the geometry may be less explicit.}

\end{lem}

\begin{pf*}
Existence and $W^{1,r}$ -stability of a right inverse of the divergence on Lipschitz-type domains can be found, for instance, in \cite[Theorem 1.4.5]{Beretal25}, where no explicit tracking of the continuity constant is provided. On domains that are star-shaped with respect to a ball, one may obtain bounds for the stability constant in terms of the chunkiness parameter; see, e.g., \cite[Chapter~III.3]{Gal94}. Extensions to domains representable as finite unions of such star-shaped subdomains can be achieved via a partition-of-unity argument; see, e.g., \cite[Section 5]{GuzSal21} and references therein. We omit the details. 
\end{pf*}

\color{black}
The subsequent relation is integral to our analysis of anisotropic meshes.

\begin{lem} \label{RTCRrel0}
For any $\bm \tau_h \in V_h^{RT}$ and $\psi_h \in V_{h}^{DC(1)}$,
\begin{align}
\displaystyle
&\int_{\Omega} \left( \bm \tau_h \cdot \bm \nabla_h \psi_{h} + \div \bm \tau_h  \psi_{h} \right) dx\notag\\
&\quad =  \sum_{F \in \mathcal{F}_h^i} \int_{F}  \{ \! \{ \bm \tau_h \} \!\}_{\omega,F} \cdot \bm n_F \Pi_F^0 [\![ \psi_{h} ]\!]_F ds + \sum_{F \in \mathcal{F}_h^{\partial}} \int_{F} ( \bm \tau_h \cdot \bm n_F) \Pi_F^0 \psi_{h} ds.  \label{wop=3a}
\end{align}

\end{lem}

\begin{pf*}
The identity is established through an elementwise application of integration by parts to $\boldsymbol{\tau}_h\cdot\nabla_h\psi_h$ and $\operatorname{div}\boldsymbol{\tau}_h\,\psi_h$, in conjunction with the definitions of the jumps and weighted averages. The standard procedural details are not included here.
\end{pf*}

\color{black}
\begin{cor} \label{RTCRrel}
Let $p^{\prime} \in [1,\infty)$. For any $\bm w \in W^{1,p^{\prime}}(\Omega)^d$ and $\psi_h \in V_{h}^{DC(1)}$,
\begin{align}
\displaystyle
&\int_{\Omega} \left( {I}_h^{RT} \bm w \cdot \bm \nabla_h \psi_{h} + \div {I}_h^{RT} \bm w  \psi_{h} \right) dx\notag\\
&\quad =  \sum_{F \in \mathcal{F}_h^i} \int_{F}  \{ \! \{ \bm w \} \!\}_{\omega,F} \cdot \bm n_F \Pi_F^0 [\![ \psi_{h} ]\!]_F ds + \sum_{F \in \mathcal{F}_h^{\partial}} \int_{F} ( \bm w \cdot \bm n_F) \Pi_F^0 \psi_{h} ds.  \label{wop=3}
\end{align}

\end{cor}

\begin{pf*}
Let $p^{\prime} \in [1,\infty)$. For any $\bm w \in W^{1,p^{\prime}}(\Omega)^d$, we set $\bm \tau_h := {I}_h^{RT} \bm w$ in \eqref{wop=3a}. Using the definition of $I_h^{RT}$ yields \eqref{wop=3}.
\end{pf*}

\begin{lem} \label{RTCR=est}
{Let $p^{\prime} \in (1,\infty)$}. For any $\bm w \in W^{1,p^{\prime}}(\Omega)^d$ and $\psi_h \in V_{h}^{DC(1)}$,
\begin{align}
\displaystyle
\left| \sum_{F \in \mathcal{F}_h^i} \int_{F}  \{ \! \{ \bm w \} \!\}_{\omega,F} \cdot \bm n_F \Pi_F^0 [\![ \psi_{h} ]\!]_F ds \right|
&\leq {C_{0a}(d,p^{\prime})} |\psi_{h}|_{p,J} \| \bm w \|_{W^{1,p^{\prime}}(\Omega)^d}, \label{wop=7} \\
\left| \sum_{F \in \mathcal{F}_h^{\partial}} \int_{F} ( \bm w \cdot \bm n_F) \Pi_F^0 \psi_{h} ds \right|
&\leq  {C_{0b}(d,p^{\prime})} |\psi_{h}|_{p,J}  \| \bm w \|_{W^{1,p^{\prime}}(\Omega)^d}, \label{wop=8}
\end{align}
where $\frac{1}{p} + \frac{1}{p^{\prime}}=1$. {Here, the constants $C_{0a}(d,p^{\prime})$ and $C_{0b}(d,p^{\prime})$ depend only on $d$ and $p^{\prime}$, and are independent of the mesh size $h$ and of the semi-regular mesh parameter $\gamma_0$.}
\end{lem}

\begin{pf*}
 Suppose that $F \in \mathcal{F}_h^i$ with {$F = \partial T_{+} \cap \partial T_{-}$, $T_{+},T_{-} \in \mathbb{T}_h$.} Using the H\"older's inequality, the weighted average and the trace inequality \eqref{trace} yields
\begin{align*}
\displaystyle
&\int_{F} \left| \{ \! \{ \bm w \} \!\}_{\omega,F} \cdot \bm n_F \Pi_F^0 [\![ \psi_{h} ]\!]_F \right | ds \\
&\quad \leq \| \omega_{T_{+},F} {\bm {w}_{+}} + \omega_{T_{-},F} {\bm {w}_{-}} \|_{L^{p^{\prime}}(F)^d} \|  \Pi_F^0 [\![ \psi_{h} ]\!]_F \|_{L^p(F)} \\
&\quad \leq \left( \omega_{T_{+},F} \| \bm {w}_{+} \|_{L^{p^{\prime}}(F)^d} + \omega_{T_{-},F} \| \bm {w}_{-} \|_{L^{p^{\prime}}(F)^d} \right)  \|  \Pi_F^0 [\![ \psi_{h} ]\!]_F \|_{L^p(F)} \\
&\quad \leq {C^{Tr}(d,p^{\prime})} \left( \| \bm {w}_{+} \|_{L^{p^{\prime}}(T_+)^d} + h_{T_+}^{\frac{1}{p^{\prime}}}  \| \bm {w}_{+} \|_{L^{p^{\prime}}(T_+)^d}^{1- \frac{1}{p^{\prime}}} | \bm {w}_{+} |_{W^{1,p^{\prime}}(T_+)^d}^{\frac{1}{p^{\prime}}} \right)  \omega_{T_{+},F} \ell_{T_+,F}^{- \frac{1}{p^{\prime}}} \|  \Pi_F^0 [\![ \psi_{h} ]\!]_F \|_{L^p(F)} \\
&\quad \quad + {C^{Tr}(d,p^{\prime})}  \left( \| \bm {w}_{-} \|_{L^{p^{\prime}}(T_{-})^d} + h_{T_{-}}^{\frac{1}{p^{\prime}}}  \| \bm {w}_{-} \|_{L^{p^{\prime}}(T_{-})^d}^{1- \frac{1}{p^{\prime}}} | \bm {w}_{-} |_{W^{1,p^{\prime}}(T_{-})^d}^{\frac{1}{p^{\prime}}} \right) \omega_{T_{-},F} \ell_{T_{-},F}^{- \frac{1}{p^{\prime}}} \|  \Pi_F^0 [\![ \psi_{h} ]\!]_F \|_{L^p(F)} \\
&\quad \leq  {C^{Tr}(d,p^{\prime})} \Bigg \{  \left( \| \bm {w}_{+} \|_{L^{p^{\prime}}(T_+)^d} + h_{T_+}^{\frac{1}{p^{\prime}}}  \| \bm {w}_{+} \|_{L^{p^{\prime}}(T_+)^d}^{1- \frac{1}{p^{\prime}}} | \bm {w}_{+} |_{W^{1,p^{\prime}}(T_+)^d}^{\frac{1}{p^{\prime}}} \right)^{p^{\prime}} \\
&\quad \quad + \left( \| \bm {w}_{-} \|_{L^{p^{\prime}}(T_{-})^d} + h_{T_{-}}^{\frac{1}{p^{\prime}}}  \| \bm {w}_{-} \|_{L^{p^{\prime}}(T_{-})^d}^{1- \frac{1}{p^{\prime}}} | \bm {w}_{-} |_{W^{1,p^{\prime}}(T_{-})^d}^{\frac{1}{p^{\prime}}} \right)^{p^{\prime}} \Bigg \}^{\frac{1}{p^{\prime}}} \\
&\quad \quad \times \left(\omega_{T_{+},F}^p \ell_{T_+,F}^{- \frac{p}{p^{\prime}}} + \omega_{T_{-},F}^p \ell_{T_{-},F}^{- \frac{p}{p^{\prime}}}\right)^{\frac{1}{p}}  \|  \Pi_F^0 [\![ \psi_{h} ]\!]_F \|_{L^p(F)}.
\end{align*}
{Using the H\"older inequality again and the fact that $h \leq 1$}, we have
\begin{align*}
\displaystyle
&\left| \sum_{F \in \mathcal{F}_h^i} \int_{F}  \{ \! \{ \bm w \} \!\}_{\omega,F} \cdot \bm n_F \Pi_F^0 [\![ \psi_{h} ]\!]_F ds \right| \\
&\quad \leq {C_1(d,p^{\prime})} \sum_{F \in \mathcal{F}_h^i}  \left(\omega_{T_{+},F}^p \ell_{T_+,F}^{- \frac{p}{p^{\prime}}} + \omega_{T_{-},F}^p \ell_{T_{-},F}^{- \frac{p}{p^{\prime}}}\right)^{\frac{1}{p}}  \|  \Pi_F^0 [\![ \psi_{h} ]\!]_F \|_{L^p(F)} \\
&\quad \quad \times \sum_{T \in \mathbb{T}_F}  \left( \| \bm {w} \|_{L^{p^{\prime}}(T)^d} + h_{T}^{\frac{1}{p^{\prime}}}  \| \bm {w} \|_{L^{p^{\prime}}(T)^d}^{1- \frac{1}{p^{\prime}}} | \bm {w} |_{W^{1,p^{\prime}}(T)^d}^{\frac{1}{p^{\prime}}} \right) \\
&\quad \leq {C_2(d,p^{\prime})} \left( \sum_{F \in \mathcal{F}_h^i}  \left(\omega_{T_{+},F}^p \ell_{T_+,F}^{- \frac{p}{p^{\prime}}} + \omega_{T_{-},F}^p \ell_{T_{-},F}^{- \frac{p}{p^{\prime}}}\right)  \|  \Pi_F^0 [\![ \psi_{h} ]\!]_F \|_{L^p(F)}^p \right)^{\frac{1}{p}} \\
&\quad \quad \times \left( \sum_{F \in \mathcal{F}_h^i}  \sum_{T \in \mathbb{T}_F}  \left( \| \bm {w} \|_{L^{p^{\prime}}(T)^d} + h_{T}^{\frac{1}{p^{\prime}}}  \| \bm {w} \|_{L^{p^{\prime}}(T)^d}^{1- \frac{1}{p^{\prime}}} | \bm {w} |_{W^{1,p^{\prime}}(T)^d}^{\frac{1}{p^{\prime}}} \right)^{p^{\prime}} \right)^{\frac{1}{p^{\prime}}} \\
&\quad \leq {C_{0a}(d,p^{\prime})} |\psi_{h}|_{p,J} \| \bm w \|_{W^{1,p^{\prime}}(\Omega)^d},
\end{align*}
which is  the inequality \eqref{wop=7} together with the weight \eqref{weight} and Jensen inequality, {where the constants $C_1(d,p^{\prime})$ and $C_2(d,p^{\prime})$ depend only on $d$ and $p^{\prime}$, and are independent of the mesh size $h$ and of the semi-regular mesh parameter $\gamma_0$.} Here, $\mathbb{T}_F$ denotes the set of the simplices in $\mathbb{T}_h$ that share $F$ as a common face.

By an analogous argument, the estimate \eqref{wop=8} holds.
\end{pf*}

\section{Main theorem} \label{sec;proof=main}
In this section, we present the discrete Sobolev inequalities on anisotropic meshes. Our goal is to prove the discrete Sobolev inequalities under the semi-regular mesh condition. The proof combines an anisotropic trace estimate, two-step affine/Piola maps, RT-interpolation stability, and a weighted discrete integration-by-parts formula.

Let $q \in (1,\infty)$ with $\frac{1}{q} + \frac{1}{q^{\prime}}=1$. For any $u \in L^q(\Omega)$ and $v \in L^{q^{\prime}}(\Omega)$,
\begin{align*}
\displaystyle
\langle u , v \rangle := \int_{\Omega} u v dx = \sum_{T \in \mathbb{T}_h}  \int_{T} u v dx.
\end{align*}
We define a functional space as
\begin{align*}
\displaystyle
X_0 &:= \left \{ \psi \in L^{q^{\prime}}(\Omega): \ \| \psi \|_{L^{q^{\prime}}(\Omega)} = 1 \right \}.
\end{align*}

\begin{thr}[Discrete $L^q$-$L^p$ Sobolev inequality] \label{main1=thr}
{Suppose that Assumption \ref{neogeo=assume} holds.} Let $p \in (1,\infty)$ and $q \in (1,\infty)$ be such that \eqref{Sobolev511},
that is $1 - \frac{d}{p} \geq - \frac{d}{q}$ and $q \leq p$. Then,
\begin{align}
\displaystyle
\| \varphi_h \|_{L^q(\Omega)} \leq {C^{dS0}(d,\Omega,p,q,\gamma_0)} |\varphi_h|_{p,V_h} \quad \forall \varphi_h \in V_h \in \{V_{h}^{DCCR}, V_h^{CR},V_{h0}^{CR} \}, \label{DSineq0}
\end{align}
where {the constant $C^{dS0}(d,\Omega,p,q,\gamma_0)$ depends only on $d$, $\Omega$, $p$, $q$ and the semi-regular mesh parameter $\gamma_0$, and is independent of $h$ and of the aspect ratios and angles of the simplices.}

\end{thr}

\begin{pf*}
Let $\varphi_h \in V_h$. Using the triangle inequality yields
\begin{align}
\displaystyle
\| \varphi_h \|_{L^q(\Omega)} 
\leq \| \varphi_h - \Pi_h^0 \varphi_h \|_{L^q(\Omega)} + \| \Pi_h^0 \varphi_h \|_{L^q(\Omega)}. \label{dissov=1}
\end{align}
{Because $\Omega$ is bounded and $q \leq p$,} by using the estimate \eqref{chap822} for each element, we obtain the following
\begin{align}
\displaystyle
\| \varphi_h - \Pi_h^0 \varphi_h \|_{L^q(\Omega)} 
&\leq {C_3(d,\Omega,p,q)} |\varphi_h|_{W^{1,p}(\mathbb{T}_h)}, \label{dissov=2}
\end{align}
{where $C_3(d,\Omega,p,q)$ depends only on $d$, $\Omega$, $p$ and $q$, and is independent of $h$ and of the aspect ratios and angles of the simplices.} Furthermore, because $\Omega$ is bounded and $q \leq p$, then $L^p(\Omega) \hookrightarrow L^q(\Omega)$ and 
\begin{align}
\displaystyle
 \| \Pi_h^0 \varphi_h \|_{L^q(\Omega)}
 &\leq {C_4(\Omega,p,q)}  \| \Pi_h^0 \varphi_h \|_{L^p(\Omega)}, \label{dissov=add}
\end{align}
{where $C_4(\Omega,p,q)$ depends only on $\Omega$, $p$ and $q$, and is independent of $h$ and of the aspect ratios and angles of the simplices.}

The $L^p$-norm of $\Pi_h^0 \varphi_h$ can be written in dual form
\begin{align*}
\displaystyle
\| \Pi_h^0 \varphi_h \|_{L^p(\Omega)} = \sup_{\psi \in X_0} \langle \Pi_h^0 \varphi_h , \psi \rangle.
\end{align*}

For any $\psi \in X_0$, we set $\psi_h := \Pi_h^0 \psi$. Then,
\begin{align*}
\displaystyle
\psi_h|_T = \frac{1}{|T|_d} \int_T \psi dx.
\end{align*}
Therefore, we have
\begin{align*}
\displaystyle
 \langle \Pi_h^0 \varphi_h , \psi \rangle
 &= \sum_{T \in \mathbb{T}_h} \Pi_T^0 \varphi_h  \int_T \psi dx 
 = \sum_{T \in \mathbb{T}_h} \Pi_T^0 \varphi_h \psi_h|_T |T|_d.
\end{align*}
On the other hand, 
\begin{align*}
\displaystyle
 \langle \Pi_h^0 \varphi_h , \psi_h \rangle = \sum_{T \in \mathbb{T}_h} \int_{T} \Pi_h^0 \varphi_h \psi_h dx =  \sum_{T \in \mathbb{T}_h} \Pi_T^0 \varphi_h \int_T \psi_h|_T dx = \sum_{T \in \mathbb{T}_h} \Pi_T^0 \varphi_h \psi_h|_T |T|_d,
\end{align*}
which leads to $ \langle \Pi_h^0 \varphi_h , \psi \rangle =  \langle \Pi_h^0 \varphi_h , \psi_h \rangle$. Using the H\"older's inequality yields
\begin{align*}
\displaystyle
\| \psi_h \|_{L^{p^{\prime}}(\Omega)}^{p^{\prime}}
&= \sum_{T \in \mathbb{T}_h} \int_T |\psi_h|^{p^{\prime}} dx =  \sum_{T \in \mathbb{T}_h} |T|_d^{1 - p^{\prime}} \left| \int_T \psi dx \right|^{p^{\prime}} \\
&\leq  \sum_{T \in \mathbb{T}_h} |T|_d^{1 - p^{\prime} + \frac{ p^{\prime}}{p}} \int_T | \psi |^{p^{\prime}} dx =  \sum_{T \in \mathbb{T}_h} \int_T |\psi|^{p^{\prime}} dx = \| \psi \|_{L^{p^{\prime}}(\Omega)}^{p^{\prime}},
\end{align*}
which leads to $\| \psi_h \|_{L^{p^{\prime}}(\Omega)} \leq  1$. 

We set $\xi_h := \psi_h - \bar{\psi}_h$, where $\bar{\psi}_h := \frac{1}{|\Omega|_d} \int_{\Omega} \psi_h dx$. Then, $\int_{\Omega} \xi_h dx = 0$. From the Bogovski\u{\i}-type lemma (Lemma \ref{lem:Bogovskii}), there exists $\bm v_0 \in W_0^{1,p^{\prime}}(\Omega)^d$ such that
\begin{align*}
\displaystyle
\div \bm v_0 &= \xi_h,\\
|\bm v_0|_{W^{1,p^{\prime}}(\Omega)^d} &\leq {C^{Bog}(d,\Omega,p^{\prime})} \| \xi_h \|_{L^{p^{\prime}}(\Omega)}.
\end{align*}
Setting $\bm v_g(\bm x) := \frac{\bar{\psi}_h}{d}(\bm x - \bm x_g)$, where $\bm x_g := \frac{1}{|\Omega|}\int_{\Omega} \bm x dx$, $\bm x \in \Omega$. Then, $\bm \nabla \bm v_g =  \frac{\bar{\psi}_h}{d} I_d$, where $I_d$ is the $d \times d$ unit matrix. This implies
\begin{align*}
\displaystyle
\div \bm v_g = \bar{\psi}_h.
\end{align*}
Furthermore, we have
\begin{align*}
\displaystyle
\| \bm v_g \|_{W^{1,p^{\prime}}(\Omega)^d}^{p^{\prime}}
&= \| \bm v_g \|_{L^{p^{\prime}}(\Omega)^d}^{p^{\prime}} + \| \bm \nabla \bm v_g \|_{L^{p^{\prime}}(\Omega)^{d \times d}}^{p^{\prime}} \\
&\leq \left( \frac{|  \bar{\psi}_h |}{d} \right)^{p^{\prime}} \diam(\Omega)^{p^{\prime}} | \Omega |_d + \left( \frac{|  \bar{\psi}_h |}{d} \right)^{p^{\prime}} \| I_d \|_F^{p^{\prime}}  | \Omega |_d \\
&\leq d^{- p^{\prime}} \left( \diam(\Omega)^{p^{\prime}} + d^{\frac{p^{\prime}}{2}}  \right) \| \psi_h \|_{L^{p^{\prime}}(\Omega)}^{p^{\prime}},
\end{align*}
where $\| \cdot \|_F$ is the Frobenius norm and we used that $\| I_d \|_F = d^{\frac{1}{2}}$ and
\begin{align*}
\displaystyle
| \bar{\psi}_h |
&\leq \frac{1}{ | \Omega |_d} \int_{\Omega} |\psi_h| dx
\leq  | \Omega |_d^{- \frac{1}{p^{\prime}}} \| \psi_h \|_{L^{p^{\prime}}(\Omega)}.
\end{align*}
We set $\bm v := \bm v_0 + \bm v_g \in W^{1,p^{\prime}}(\Omega)^d$. Then, it holds that
\begin{align*}
\displaystyle
\div \bm v &= \div \bm v_0 + \div  \bm v_g = \psi_h, \\
\| \bm v \|_{W^{1,p^{\prime}}(\Omega)^d}
&\leq \| \bm v_0 \|_{W^{1,p^{\prime}}(\Omega)^d} + \| \bm v_g \|_{W^{1,p^{\prime}}(\Omega)^d} \\
&\leq {C_5(d,\Omega,p^{\prime})} \| \psi_h \|_{L^{p^{\prime}}(\Omega)} \leq {C_5(d,\Omega,p^{\prime})},
\end{align*}
where we used that
\begin{align*}
\displaystyle
\| \bm v_0 \|_{W^{1,p^{\prime}}(\Omega)^d}
&\leq {C_6(d,\Omega,p^{\prime})} \| \xi_h \|_{L^{p^{\prime}}(\Omega)}
\leq {C_6(d,\Omega,p^{\prime})} \left( \| \psi_h \|_{L^{p^{\prime}}(\Omega)} +  \| \bar{\psi}_h \|_{L^{p^{\prime}}(\Omega)} \right) \\
&\leq {C_7(d,\Omega,p^{\prime})}  \| \psi_h \|_{L^{p^{\prime}}(\Omega)}.
\end{align*}
{Here, $C_5(d,\Omega,p^{\prime})$, $C_6(d,\Omega,p^{\prime})$ and $C_7(d,\Omega,p^{\prime})$ depend only on $d$, $\Omega$ and $p^{\prime}$ and is independent of $h$ and of the aspect ratios and angles of the simplices.}

It holds that
\begin{align*}
\displaystyle
\div I_h^{RT} \bm v = \Pi_h^0 \div \bm v =  \Pi_h^0 \psi_h = \psi_h.
\end{align*}
Using the stability of the RT interpolation (Corollary \ref{RT=sta}) yields
\begin{align*}
\displaystyle
\|  I_h^{RT} \bm v \|_{L^{p^{\prime}}(\Omega)^d}
&\leq {C^{RTsta}(d,p^{\prime},\gamma_0)} \| \bm v \|_{W^{1,p^{\prime}}(\Omega)^d} 
\leq {C_8(d,\Omega,p^{\prime},\gamma_0)}  \| \psi_h \|_{L^{p^{\prime}}(\Omega)} \leq {C_8(d,\Omega,p^{\prime},\gamma_0)},
\end{align*}
{where $C_8(d,\Omega,p^{\prime},\gamma_0)$ depends only on $d$, $\Omega$, $p^{\prime}$ and the semi-regular mesh parameter $\gamma_0$, and is independent of $h$ and of the aspect ratios and angles of the simplices.}

From Lemma \ref{RTCRrel} and the definition of $\Pi_h^0$, we obtain
\begin{align*}
\displaystyle
 \langle \Pi_h^0 \varphi_h , \psi \rangle 
 &=  \langle \Pi_h^0 \varphi_h , \psi_h \rangle = \sum_{T \in \mathbb{T}_h} \int_T \varphi_h  \psi_h dx \\
 &= \sum_{T \in \mathbb{T}_h} \int_T \varphi_h \div I_h^{RT} \bm v dx \\
 &= - \sum_{T \in \mathbb{T}_h} \int_T {I}_h^{RT} \bm v \cdot \bm \nabla \varphi_{h} dx + \sum_{F \in \mathcal{F}_h^i} \int_{F}  \{ \! \{ \bm v \} \!\}_{\omega,F} \cdot \bm n_F \Pi_F^0 [\![ \varphi_{h} ]\!]_F ds + \sum_{F \in \mathcal{F}_h^{\partial}} \int_{F} ( \bm v \cdot \bm n_F) \Pi_F^0 \varphi_{h} ds.
\end{align*}
Using Lemma \ref{RTCR=est} and the above results yields
\begin{align*}
\displaystyle
\left|  \langle \Pi_h^0 \varphi_h , \psi \rangle  \right|
&\leq {C_8(d,\Omega,p^{\prime},\gamma_0)} |\varphi_h|_{W^{1,p}(\mathbb{T}_h)} + {C_9(d,p^{\prime})} |\varphi_{h}|_{p,J} \leq {C_{10}(d,\Omega,p^{\prime},\gamma_0)} | \varphi_h |_{p,V_h},
\end{align*}
which leads to
\begin{align}
\displaystyle
\| \Pi_h^0 \varphi_h \|_{L^p(\Omega)}
&\leq {C_{10}(d,\Omega,p^{\prime},\gamma_0)} | \varphi_h |_{p,V_h}. \label{dissov=3}
\end{align}
{Here, $C_9(d,p^{\prime})$ depends only on $d$ and $p^{\prime}$, and $C_{10}(d,\Omega,p^{\prime},\gamma_0)$ depends only on $d$, $\Omega$, $p^{\prime}$ and the semi-regular mesh parameter $\gamma_0$. They are independent of $h$ and of the aspect ratios and angles of the simplices.}

Gathering the inequalities \eqref{dissov=1}, \eqref{dissov=2}, \eqref{dissov=add} and \eqref{dissov=3}, we conclude the target estimate.
\end{pf*}

\begin{rem}
When $q \> p$, the proof of Theorem \ref{main1=thr} can not be applied. The $L^q$-norm of $\Pi_h^0 \varphi_h$ can be written in dual form
\begin{align*}
\displaystyle
\| \Pi_h^0 \varphi_h \|_{L^q(\Omega)} = \sup_{\psi \in X_0} \langle \varphi_h , \psi_h \rangle.
\end{align*}
From the bilateral estimate,
\begin{align*}
\displaystyle
\| \Pi_h^0 \varphi_h \|_{L^q(\Omega)}
&\leq \| \varphi_h \|_{L^p(\Omega)}  \sup_{\psi \in X_0} \| \psi_h \|_{L^{p^{\prime}}(\Omega)}.
\end{align*}
Because $\psi_h|_T$ is a constant, 
\begin{align*}
\displaystyle
\| \psi_h \|_{L^{p^{\prime}}(T)}
&\leq c |T|_d^{\frac{1}{p^{\prime}} - \frac{1}{q^{\prime}}} \| \psi_h \|_{L^{q^{\prime}}(T)} \leq c  |T|_d^{\frac{1}{q} - \frac{1}{p}} \| \psi_h \|_{L^{q^{\prime}}(T)}.
\end{align*}
When an isotropic mesh is used, it is expressed as $|T|_d \approx h_T^d$. Then, 
\begin{align*}
\displaystyle
\| \psi_h \|_{L^{p^{\prime}}(T)}
&\leq c  h_T^{d \left( \frac{1}{q} - \frac{1}{p} \right)} \| \psi_h \|_{L^{q^{\prime}}(T)} \to \infty \quad \text{as $h_T \to 0$}.
\end{align*}
\end{rem}

\section{Concluding remarks} \label{Con=rem}

This paper presents discrete $L^q-L^p$ Sobolev inequalities for nonconforming finite elements under a semi-regular mesh condition, with constants that remain independent of the angle and aspect ratio of simplices. The proof integrates anisotropy-aware trace/projection estimates with RT interpolation and face-weighted flux control. These findings extend the classical shape-regular theory to anisotropic partitions and directly support stability and a priori analyses for CR and Nitsche schemes.

However, when $q \> p$, it remains to show the following restricted discrete Sobolev inequality, see \cite[Lemma 4]{Ish25a}. For that purpose, we imposed a weak elliptic regularity assumption to obtain a discrete Sobolev inequality: Let $p \in [2 , \infty )$ if $d=2$, or $p \in [2,6]$ if $d=3$, then $p^{\prime} \in (\frac{2d}{d+2} , 2]$. We assume that, for any $g \in L^{p^{\prime}}(\Omega)$, the variational problem 
\begin{align*}
\displaystyle
\int_{\Omega} \bm \nabla z \cdot  \bm \nabla w dx = \int_{\Omega} g w dx \quad \forall w \in H_0^1(\Omega) 
\end{align*}
has a unique solution $z \in H_0^1(\Omega)$ that belongs to $W^{2,p^{\prime}}(\Omega)$ and the elliptic regularity estimate
\begin{align*}
\displaystyle
\| z \|_{W^{2,p^{\prime}}(\Omega)} \leq c \| g \|_{L^{p^{\prime}}(\Omega)}
\end{align*}
holds. In addition to the weak elliptic regularity assumption, we impose that Assumptions \ref{neogeo=assume} and there exists a positive constant {$C^{(p,2)}$} independent of $T$ and $h$ such that
\begin{align*}
\displaystyle
\max_{T \in \mathbb{T}_h} \left( {|{T}|_d}^{\frac{1}{p} - \frac{1}{2}} h_T \right) \leq C^{(p,2)}.
\end{align*}
Then, 
\begin{align*}
\displaystyle
\| \varphi_h \|_{L^p(\Omega)} \leq c |\varphi_h|_{H^1(\mathbb{T}_h)} \quad \forall \varphi_h \in V_{h0}^{CR}. 
\end{align*}
In this statement, when $p=2$, the domain is required to be convex to satisfy the weak elliptic regularity assumption. However, in Theorem \ref{main1=thr}, the assumption that the domain is convex is removed. This means that the weak elliptic regularity assumption may be able to be removed. This issue is left for future work.

\end{document}